\documentclass[12pt]{amsart}
\usepackage{amsmath}
\usepackage{amssymb}
\usepackage{amsthm}
\usepackage{graphicx}
\usepackage{enumitem}
\usepackage{amsfonts}
\usepackage{amssymb}
\usepackage{bbm}
\usepackage{cases}
\usepackage{cite}
\usepackage{ifthen}
\usepackage{geometry}
\usepackage{hyperref}
\usepackage{xcolor}

\newcommand{\bS}{\mathbf{S}}
\newcommand{\bs}{\mathbf{s}}
\newcommand{\by}{\mathbf{y}}
\newcommand{\bu}{\mathbf{u}}
\newcommand{\btau}{\ensuremath{\boldsymbol{\tau}}}
\newcommand{\bT}{\mathbf{T}}
\newcommand{\bB}{\mathbf{B}}
\newcommand{\bv}{\mathbf{v}}
\newcommand{\bK}{\mathbf{K}}
\newcommand{\bt}{\mathbf{t}}
\newcommand{\Ai}{\mathrm{Ai}}

\newcommand{\bx}{\mathbf{x}}
\newcommand{\R}{\mathbb{R}}

\newcommand{\Aistat}{\mathrm{Ai}_{\mathrm{stat}}}
\newtheorem{theorem}{Theorem}[section]
\newtheorem{corollary}[theorem]{Corollary}
\newtheorem{lemma}[theorem]{Lemma}
\newtheorem{proposition}[theorem]{Proposition}
\newtheorem{remark}[theorem]{Remark}

\title{$t^{1/3}$ Fluctuation around the shock of TASEP with random initial condition}

\begin{document}

\begin{abstract}
The totally asymmetric exclusion process (TASEP) is one of the solvable models in the KPZ universality class. When TASEP starts with product Bernoulli measure with smaller density on the left of the origin, it presents shocks in the evolution. For a long time, it has been known that fluctuations are the product of Gaussians on the scale $t^{1/2}$ due to initial randomness. In this paper, we will describe how to see the $t^{1/3}$ fluctuations for these initial conditions.

\end{abstract}
\author[X. Zhang]{Xincheng Zhang}
\address{Xincheng Zhang, University of Toronto, Department of Mathematics,
45 St George Street,
Toronto, ON M5S 2E5
, Canada}
\email{xincheng.zhang@mail.utoronto.ca}

\maketitle
\tableofcontents

\section{Introduction}
\subsection{Model and previous results}
In this paper, we study the fluctuations of the totally asymmetric exclusion process(TASEP) around its shock. TASEP is a continuous time interacting particle system on an integer lattice. The dynamic is as follows. Each particle waits an exponential amount of time with rate 1 and then jumps to the right if the right site is empty. If there is a particle in the right site, then the jump is blocked and the exponential clock runs again. For a rigorous construction for an infinite number of particles, see \cite{liggetIPS} The model has been extensively studied because of its physical importance and solvability. We start with a short introduction to the model and some of its important facts.

There are several observables for TASEP. Let $X_t(n)$ be the position of the $n$-th particles at time $t$. For initial conditions, we use the convention that:
\begin{equation}
 \cdots < X_0(2)<X_0(1)<0\leq X_0(0)<X_0(-1)<X_0(-2)<\cdots.
 \end{equation} Equivalently, we can use the occupation variables, $\eta_t(x),x\in \mathbb{Z}$, which is 1 if there is a particle at position $x$  at time $t$ and $0$ if there is no particle at position $x$ at time $t$. There is also the height representation $h(t,x)$,
 \begin{equation}
  h(t,x)=  \begin{cases}
      &2J_{0,t}+\sum_{i=1}^x(1-2\eta_t(i)),\quad x\geq 1\\
      &2J_{0,t}, \quad x=0\\
      &2J_{0,t}-\sum_{i=x+1}^0(1-2\eta_t(i)),\quad x\leq -1,
  \end{cases}
  \end{equation}
  where $J_{0,t}$ is the number of particles that have crossed $0-1$ up to time $t$. We will use the notation interchangeably. 
  
  Under hydrodynamic scaling, the particle density function of TASEP satisfies Burger's equation\cite{FR91}\cite{ST98}. The macroscopic density $u(t,x)$ is defined to be:
 \begin{equation}
 \int_a^b u(t,x)dx = \lim_{t\to\infty}\sum_{i=\lfloor a\varepsilon^{-1}\rfloor}^{\lfloor b\varepsilon^{-1}\rfloor}\varepsilon\eta_{\varepsilon^{-1}t}(i),
 \end{equation}
 and the burger's equation is:
 \begin{equation}
 \partial_t u+\partial_x(u(1-u))=0
 \end{equation}
 with the initial condition being the corresponding density limit for the TASEP initial condition. 
 
If we start with a initial condition such that the particle density on the left of the origin is $\rho_-$ and the density on the right-hand side is $\rho_+$, more precisely,
\begin{equation}
 \rho_{\pm}  = \lim_{N\to\infty} \frac{\sum_{i=0}^{\pm N}\eta_0(i)}{N},
\end{equation}
then the density function solves the Burger's equation starting with the initial condition
\begin{equation}
u(0,x)=\rho_-\mathbf{1}_{x<0}+\rho_+\mathbf{1}_{x\geq0}.
\end{equation}
When $\rho_-<\rho_+$, there is a shock in the solution of Burger's equation, where the shock means that there is a discontinuity in the density. The characteristic line of the Burgers equation with constant density $\rho$ is given by $x(t)=x(0)+(1-2\rho)t$. The line of the shock that satisfies the entropy condition is given by $x(t) = (1-\rho_+-\rho_-)t$. 

We are interested in TASEP starting with the initial conditions being the product of independent Bernoulli random variables, with $\eta(x)\sim \mathrm{Ber}(\rho_-)$ for $x\leq 0$ and $\eta(x)\sim \mathrm{Ber}(\rho_+)$ for $x> 0$, with $\rho_-<\rho_+$. We are interested in the fluctuations of TASEP around the place of the shock. From the hydrodynamic scaling, we know that the place of the shock travels alone the line $x(t)=(1-\rho_+-\rho_-)t$. Thus, we are interested in the normalized height function around $(1-\rho_+-\rho_-)$. Another way of studying the shock is to study the position of the particle which ``identifies" the shock. It is not easy to define such a particle. The position of the shock in TASEP is identified with a second class particle\cite{liggett1999sip}. A second class particle follows the same dynamic as TASEP particles, but whenever a first class particle tries to jump and the second class particle is in the way, they will switch positions. Let $X_t$ denote the position of the second class particle that is initially placed at 0. With the above initial condition, $X_t$ moves with average speed $1-\rho_--\rho_+$ and its fluctuation is Gaussian on the $t^{1/2}$ scale\cite{ferrari1992}, see also \cite{GP90,DKPS89}. The method is purely probabilistic.

During the past 25 years, the integrable structure of TASEP has been extensively studied. The full space TASEP was first solved by \cite{Schutz97} using coordinate Bethe ansatz. The formula given in \cite{Schutz97} is not conducive to asymptotic analysis. Later, \cite{Sas05} \cite{BFPS07} rewrite the determinant and formula using the biorthogonalization problem, from which the transition probability can be derived. Under some special initial conditions, the biorthogonalization problem is solved \cite{BFP06}\cite{BFP08}\cite{BFPS07}. In \cite{MQR2021}, the biorthogonalization problem was solved with a general initial condition. There is also extensive study on the last-pass percolation model, which is a closely related model to TASEP \cite{J2000}. The integrable structures in LPP is also a method for studying TASEP. Stationary TASEP is studied in \cite{bfp2010}. In-homogeneous TASEP is studied in \cite{BBP05}\cite{BCCurrentfluctuationsforTASEP}, both are using the relations to the LPP model. For LPP model with general initial condition, see \cite{CLW16}.

        With all integral probability methods, TASEP starting with product Bernoulli measure with density $\rho_-$ on one side and density $\rho_+$ on the other side is completely studied in \cite{BCCurrentfluctuationsforTASEP}\cite{CFP10}, which gave the proof of the conjectures given in \cite{PS02}, including the case that $\rho_-<\rho_+$ where shock is present. The fluctuation is the product of two normal distributions. This particular case is not the main point of \cite{BCCurrentfluctuationsforTASEP}, since there is a relation between the place of the second class particle and the currents of the particles, for a good survey on it, see \cite{Pablo18}, and for a detailed proof, see \cite{FN24}. Thus, one important heuristic is that if the second class particle position is given by the difference of two random variables, then the fluctuation of the height function around the place of the shock will be the product of two distributions corresponding to the random variables above. 

        In-homogeneous TASEP starting with a deterministic initial condition is also studied. If we start with a periodic initial condition with density $\rho_-$ on the left of the origin and $\rho_+$ on the right of the origin, we see that the fluctuation is the product of two Tracy-Widom GOE distributions living on the $t^{1/3}$ scale\cite{fn2013}. The second class particle is the difference of two random variables, having an independent GOE distribution \cite{FGN19}\cite{FN24}. For the process under critical scaling, see \cite{FN15,qm2018}. The shock behavior under a combination of narrow wedge initial conditions and flat initial conditions is also studied in \cite{fn2013}. For more studies related to the shock behavior in TASEP, see \cite{FN20,FB22,Fer18,chen24,Nej18}.The analog object of shock in the LPP world is the competing interface, which is also studied in \cite{FP05,FMP09,RV23, SS23}.
        \subsection{New results and methods}All examples in the deterministic initial condition have fluctuations on the scale $t^{1/3}$, which come from the evolution of TASEP. The distribution depends on the initial condition around where the characteristic line starts. The one-point distribution of stationary TASEP along the characteristic line is the Baik-Rains distribution, on the scale $t^{1/3}$. Thus, in the case where we start in-homogeneous TASEP with product Bernoulli measures, the fluctuation of the shock hidden under Gaussian should be the product of the Baik-Rains distributions. The purpose of this paper is to reveal this.

        The method is to make a fine study of exact formulas. To see the fluctuation on the $t^{1/3}$ scale, we need to condition on the value of the Gaussian random variables. Meanwhile, we want to maintain the randomness from the initial condition to see the Baik-Rain distribution. The key point of the analysis is that the probability distribution of TASEP depends only on the initial condition around the start of the characteristics on the scale $t^{2/3}$, thus, we are able to achieve both of the above goals. 

        \subsection{Main results}
        Throughout the paper, we define 
        \begin{equation}
            \chi_- = \rho_-(1-\rho_-), \quad \chi_+=\rho_+(1-\rho_+).
        \end{equation}Before we state the new results, we state an old theorem as a comparison. The results for the second class particle appear earlier in \cite{ff94annalsofprobability}, the following theorem is in a form similar to what we are going to prove.
\begin{theorem}\label{thm:bcnormtheorem}\cite{BCCurrentfluctuationsforTASEP}
Let $\rho_-<\rho_+$ and $v = 1-\rho_--\rho_+$, and $\overline{h} = (1-\rho_--\rho_++2\rho_-\rho_+)$ 
\begin{equation}
\begin{split}
\lim_{\varepsilon\to 0}\mathbb{P}&(\frac{h(\varepsilon^{-3/2},v\varepsilon^{-3/2})-\overline{h}\varepsilon^{-3/2}}{\varepsilon^{-3/4}}\geq -(\rho_+-\rho_-)^{1/2}x)=G( (4\chi_+)^{-1/2}x)G( (4\chi_-)^{-1/2}x)
\end{split}
\end{equation}
where $G$ is the standard Gaussian distribution.
\end{theorem}
This theorem indicates that the rescaled height function is the maximum of two Gaussian random variables. It is easy to see where the randomness comes from: it is the initial height at $x = (\rho_--\rho_+)\varepsilon^{-3/2}$ and $x= (\rho_+-\rho_-)\varepsilon^{-3/2}$, with fluctuations on the scale of $\varepsilon^{-3/4}$, which is the limit of the following two random variables when $\bt=1$:
\begin{equation}\label{eqn:NlNrdef}
    \begin{split}
        N_r = -(\rho_+-\rho_-)\rho_+\varepsilon^{-3/2}\bt+\sum_{0\leq x<(\rho_+-\rho_-)\varepsilon^{-3/2}\bt}\eta_0(x)\\
        N_l = -(\rho_+-\rho_-)\rho_-\varepsilon^{-3/2}\bt+\sum_{(\rho_--\rho_+)\varepsilon^{-3/2}\bt\leq x<0}\eta_0(x).\\
    \end{split}
\end{equation} 
 Now we want to move around the shock place on the same scale, $\varepsilon^{-3/4}$, to make the mean contribution from this shift cancel the difference between two Gaussians above. Now we look at the height function at this place, we should see the fluctuations coming from the competition from the TASEP evolution, on $\varepsilon^{-1/2}$ scale. 
\begin{theorem}\label{thm:productat0}
In the same notation as in Theorem \eqref{thm:bcnormtheorem}. Define $\tilde{\bs}_- = \chi_-^{-2/3}(1-2\rho_-)(1-2\rho_-)\bs, \quad \tilde{\bs}_+ = \chi_+^{-2/3}(1-2\rho_+)\bs$. We have 
\begin{equation}
\begin{split}
\lim_{\varepsilon \to 0}\mathbb{P}&(\frac{h(\varepsilon^{-3/2}\bt,v\varepsilon^{-3/2}\bt-\frac{N_r+N_l}{\rho_+-\rho_-})-\overline{h}\varepsilon^{-3/2}\bt-\frac{(2\rho_--1)N_r+(2\rho_+-1)N_l}{\rho_+-\rho_-}}{\varepsilon^{-1/2}}\geq -2\bs)\\
&=F_{0}(\bt^{-1/3}\tilde{\bs}_-)F_{0}(\bt^{-1/3}\tilde{\bs}_+)
\end{split}
\end{equation}
where $F_{w}(s)$ is a family of distribution indexed by $w \in \mathbb{R}$, which gives the one-point distribution of Airy stationary process\cite{FS06}\cite{BFP08}. $F_0(s)$ is called the Baik-Rains distribution \cite{BR00baikrains}.  Its complete definition will be given in Section (\ref{sec:identifylimitingkernel}).
\end{theorem}
Using the relation between height representation and particle representation,
\begin{equation}
\mathbb{P}(h(t,x-y)\geq x+y) = \mathbb{P}(X_t(y)\geq x-y).
\end{equation}
We can transform the result into particle representation.
\begin{corollary}\label{Cor:productparticleat0}
\begin{equation}
\begin{split}
\lim_{\varepsilon\to 0}\mathbb{P}&(X_t(\rho_+\rho_-\varepsilon^{-3/2}\bt+\frac{\rho_-N_r+\rho_+N_l}{\rho_+-\rho_-}-\varepsilon^{-1/2}\bs)\geq (1-\rho_+-\rho_-)\varepsilon^{-3/2}\bt-\frac{N_r+N_l}{\rho_+-\rho_-})\\
&=F_{0}(\bt^{-1/3}\tilde{\bs}_-)F_{0}(\bt^{-1/3}\tilde{\bs}_+).
\end{split}
\end{equation}
\end{corollary}
We can upgrade it to a process version, by looking at the height function at $t = \varepsilon^{-3/2}\bt+\varepsilon^{-1}\bx$. It turns out to be the maximum of two $\text{Airy}_{\text{stat}}$ processes.
\begin{theorem}\label{thm:processversionheightfunction}
In the same notation as in Theorem \eqref{thm:bcnormtheorem}, define $y_\varepsilon = \varepsilon^{-3/2}\bt+\varepsilon^{-1}\bx$. Let $\tilde{\bx}_-=(\rho_--\rho_+)2^{-1}\chi_-^{-1/3}\bx,\quad \tilde{\bx}_+=(\rho_+-\rho_-)2^{-1}\chi_+^{-1/3}\bx.$
\begin{equation}
\begin{split}
\lim_{\varepsilon \to 0}\mathbb{P}&(\frac{h(y_\varepsilon,vy_\varepsilon-\frac{N_r+N_l}{\rho_+-\rho_-})-\overline{h}y_\varepsilon-\frac{(2\rho_--1)N_r+(2\rho_+-1)N_l}{\rho_+-\rho_-}}{\varepsilon^{-1/2}}\geq -2
\bs)\\
&=F_{\bt^{-2/3}\tilde{\bx}_-}(\bt^{-1/3}\tilde{\bs}_-)F_{\bt^{-2/3}\tilde{\bx}_+}(\bt^{-1/3}\tilde{\bs}_+).
\end{split}
\end{equation}    
\end{theorem}
The following is the particle representation.
\begin{corollary}\label{cor:processversionparticle}
    \begin{equation}
\begin{split}
\lim_{\varepsilon\to 0}\mathbb{P}&(X_t(\rho_+\rho_-y_\varepsilon+\frac{\rho_-N_r+\rho_+N_l}{\rho_+-\rho_-}-\varepsilon^{-1/2}\bs)\geq (1-\rho_+-\rho_-)y_\varepsilon-\frac{N_r+N_l}{\rho_+-\rho_-})\\
&=F_{\bt^{-2/3}\tilde{\bx}_-}(\bt^{-1/3}\tilde{\bs}_-)F_{\bt^{-2/3}\tilde{\bx}_+}(\bt^{-1/3}\tilde{\bs}_+).
\end{split}
\end{equation}
\end{corollary}
We can also state the version for the second class particle.
\begin{theorem}\label{thm:secondclassparticle}
Let $X^{\mathrm{2nd}}$ denote the position of the second class particle initially placed at the origin, with the same notation as in Theorem \eqref{thm:processversionheightfunction},
    \begin{multline}
        \lim_{\varepsilon \to 0}\mathbb{P}(X^{\mathrm{2nd}}(y_\varepsilon)\geq vy_\varepsilon-\frac{N_r+N_l}{\rho_+-\rho_-}+\bs\varepsilon^{-1/2}) \\= \mathbb{P}(\chi_-^{2/3}\Aistat(\bt^{-2/3}\tilde{\bx}_-)-\chi_+^{2/3}\Aistat'(\bt^{-2/3}\tilde{\bx}_+)\geq 2(\rho_+-\rho_-)\bt^{-1/3}\bs)
    \end{multline}
    where $\mathrm{Ai}_{\mathrm{stat}},\mathrm{Ai}_{\mathrm{stat}}'$ are two independent Airy stationary process defined in \cite{FS06}\cite{BFP08}. Its one-point distribution is $F_w(s)$, which is introduced in Theorem \eqref{thm:productat0}.
\end{theorem}

\subsection{Outline}Theorem \eqref{thm:productat0}\eqref{Cor:productparticleat0} are the special case of \eqref{thm:processversionheightfunction},\eqref{cor:processversionparticle}, and \eqref{thm:processversionheightfunction} are equivalent to \eqref{cor:processversionparticle}, so we will prove Corollary \eqref{cor:processversionparticle}. In Section 3, we develop all the ingredients that we need for the TASEP distribution function. It is an extension of \cite{MQR2021} into the case with density $\rho$. In Section 4, we will explain how the Fredholm determinant factor into two pieces in the limit, each piece representing the fluctuations traveling from one of the characteristics lines, so the question boils down to analyzing two kernels. In Section 5, we will show how these two kernels converge to the operator we want and how to identify the limiting object. In Section 6, we will prove all the convergence results used in Section 5. In Section 7, we prove Theorem \eqref{thm:secondclassparticle}. Thanks to a nice theorem in \cite{FN24}, which we can directly apply to conclude the place about the second class particle. In the appendix, we prove the trace norm bounds on the TASEP kernel. 
\subsection{Acknowledgements}
The author would like to thank Jeremy Quastel for suggesting this problem and for many discussions on paper\cite{MQR2021}, which enable the author to perform such an analysis. The problem was asked by Pablo Ferrari in an online talk of Patrick Ferrari. The research is partially supported by NSERC.
\section{TASEP distribution function}
In this section, we derive the TASEP distribution function starting from any deterministic initial condition. The formula in \cite{MQR2021} is only for the configuration with density $1/2$. We extend the formula to the case with density $\rho$. Note that Lemma (\ref{splittinglemma}) is new and not in \cite{MQR2021}.

We start from the biorthogonalization characterization of the TASEP distributions.

\begin{theorem}\cite{b}\label{TASEPformula}
Suppose that TASEP starts with particles labeled $1,2, \cdots$. Then for $t>0$ we have 
\begin{equation}
\mathbb{P}(X_t(n)>a) = \det(I-\overline{1}^aK_t\overline{1}^a)_{l^2(\mathbb{Z})}
\end{equation}
where $\overline{1}^a(x) = \mathbf{1}_{x\leq a}$, $\det$ is the Fredholm determinant, and $K_t$ is a integral kernel,
\begin{equation}
K_t(x_1,x_2) = \sum_{k=0}^{n-1}\Psi^n_k(x_1)\Phi^n_k(x_2),
\end{equation}
where 
\begin{equation}
\Psi^n_k(x) = \frac{1}{2\pi i}\int_{\Gamma_0}dw\frac{(1-w)^k}{w^{x+k+1-X_0(n-k)}}e^{t(w-1)}
\end{equation}
where $\Gamma_0$ is any positive oriented simple loop including the pole at $w = 0$ but not the one at $w = 1$. The function $\Phi^n_k(x),k=0,\cdots,n-1$, are defined implicitly by
\begin{enumerate}
\item The biorthogonalization relation: 
\begin{equation}\label{eqn:biorthogonalizationrelation}
    \sum_{x\in \mathbb{Z}}\Psi^n_k(x)\Phi^n_k(x) = \mathbf{1}_{k=l};
\end{equation}
\item  $\Phi^n_k(x)$ is a polynomial of degree at most $n-1$ in $x$ for each $k$.
\end{enumerate}
\end{theorem}

The formula obtained in the original paper is for the multipoint distribution, but in this paper, we will only use the one-point distribution formula. Notice that the Fredholm determinant is invariant under conjugation, i.e. $\det(I-A) = \det(I-BAB^{-1})$. For TASEP with initial density $\rho$, we need to conjugate the kernel with $(1-\rho)^{x_1-x_2}$ so that under correct scaling for $n, a, t$, the kernel will converge to a limiting kernel. In \cite{MQR2021}, the paper split the kernel into several parts that have two advantages: 1. The effect of the initial condition is compressed into a hitting operator, which has a good probabilistic interpretation. 2. Each part converges to a limiting operator. The work done is \cite{MQR2021} is for TASEP with initial density $1/2$. We will follow their method and idea, coming up with a formula for TASEP with initial density $\rho$. We first conjugate the kernel and write it as
\begin{equation}
K_{\rho,t}(x_1,x_2) = (1-\rho)^{x_1-x_2}K_t(x_1,x_2) = \sum_{k=0}^{n-1}\Psi^n_{\rho,k}(x_1)\Phi^n_{\rho,k}(x_2),
\end{equation}
where 
\begin{equation}
\begin{split}
&\Psi^n_{\rho,k}(x) = \frac{1}{2\pi i}\int_{\Gamma_0}dw(\frac{1-\rho}{\rho})^k \frac{(1-\rho)^{x-X_0(n-k)}(1-w)^k}{w^{x+k+1-X_0(n-k)}}e^{t(w-1)};\\
&\Phi^n_{\rho,k}(x) = (1-\rho)^{X_0(n-k)-x}(\frac{\rho}{1-\rho})^{k}\Phi^n_{k}(x).
\end{split}
\end{equation}
Now we want to further rewrite these two functions as a composition of other operators. Before that, we need to introduce some operators. Define
\[Q_\rho(x,y) = (1-\rho)^{x-y-1}\rho\mathbf{1}_{x>y}.\] It is easy to see that $x-y$ follows the Geom$[\rho]$ distribution, therefore the operator $Q_\rho(x,y)$ can be thought of as the one-step transition probability of a Geom$[\rho]$ random walk jumping to the left; we will use $B(x)$ to denote such a random walk. Meanwhile, its adjoint $Q^*_\rho(x,y) = Q_\rho(y,x)$, is the transition probability of a Geom$[\rho]$ random walk jumping to the right, which we denote as $B^*(x)$. Its inverse is 
\[Q^{-1}_\rho =\frac{1}{\rho}\mathbf{1}_{x=y-1}-\frac{1-\rho}{\rho}\mathbf{1}_{x=y} = \mathbf{1} + \frac{1}{\rho}\nabla^+,\] where $\nabla^+f(x) = f(x+1)-f(x)$. Then define 
\[R_{\rho,t}(x,y) = e^{-t}\frac{(1-\rho)^{x-y}t^{x-y}}{(x-y)!}\mathbf{1}_{x\geq y}=\frac{1}{2\pi i}\int_{\Gamma_0}dw \frac{(1-\rho)^{x-y}e^{t(w-1)}}{w^{x-y+1}} = e^{-\rho t+(1-\rho)t\nabla^-}(x,y).\]\\
Now we can write $\Psi^n_{\rho,k}$ as follows:
\begin{equation}\label{psiformula}
\Psi^n_{\rho,k} = R_{\rho,t}Q^{-k}_{\rho}\delta(X_0(n-k)).
\end{equation}
It can be verified directly using the contour formula that it is the same as the formula defined in Theorem (\ref{TASEPformula}). Meanwhile, the function $\Phi_{\rho,k}^n$, which is defined by the biorthogonal condition, satisfies the following difference equation:
\begin{theorem}\label{hnkdef}\cite{MQR2021}
Fix $0\leq k<n$ and consider particles at $X_0(1)>X_0(2)>\cdots>X_0(r)$. Let $h^n_k(l,z)$ to be the unique solution to the initial-boundary value problem for the backwards heat equation
\begin{numcases}{}
                (Q^*_\rho)^{-1}h^n_k(l,z) = h^n_k(l+1,z), l<k, z\in \mathbb{Z}\label{eq1}\\
               h^n_k(k,z) = \rho(1-\rho)^{X_0(n-k)-z-1},z\in \mathbb{Z}\label{eq2}\\ 
               h^n_k(l,X_0(n-l)) = 0, l<k.\label{eq3}
\end{numcases}
Then the function $\Phi^n_k$ from Theorem \ref{TASEPformula} are given by
\begin{equation}
\Phi^n_{\rho,k}(z)= \sum_{y\in \mathbb{Z}}\frac{1-\rho}{\rho}h^n_k(0,y)R^{-1}_{t,\rho}(y,z).
\end{equation}
\end{theorem}  
The theorem is adapted to the density $\rho$  case, so we will write the proof here.
\begin{proof}
Using (\ref{eq1}), we have $$h^n_k(l+1,z) = \frac{1}{\rho}h^n_k(l,z-1)-\frac{1-\rho}{\rho}h^n_k(l,z).$$
Using telescoping sum, we have 
\begin{equation}\label{relation}
\frac{1}{(1-\rho)^{n-1}\rho}h^n_k(l,z-n)-\frac{1-\rho}{\rho}h^n_k(l,z) = \sum_{j=0}^{n-1}\frac{1}{(1-\rho)^j}h^n_k(l+1,z-j).
\end{equation}
By relation (\ref{eq3}), we set the value at $(l,X_0(n-l))$ to be 0. Thus, knowing the value $h^n_k(l+1,\cdot)$, we can uniquely solve the value at $h^n_k(l,\cdot)$, for $l<k$.

Now we first prove that $(1-\rho)^zh^n_k(0,z)$ is a polynomial of degree at most $k$ in the variable $z$ by induction. First, by (\ref{eq2}), $(1-\rho)^zh^n_k(k,z)$ is a polynomial of degree 0. Now assume $(1-\rho)^zh^n_k(l,z)$ is a polynomial of degree at most $k-l$ for some $0<l\leq k$. For $z \geq X_0(n-l+1)$, using (\ref{relation}), we have
\[-\frac{1-\rho}{\rho}h^n_k(l-1,z) = \sum_{j=0}^{z-X_0(n-l+1)-1}\frac{1}{(1-\rho)^j}h^n_k(l,z-j).\] By multiplying $(1-\rho)^z$ on both sides, the right side is the sum of at most $z$ constant terms of polynomials of degree at most $k-l$, thus $h^n_k(l-1,z)$ is a polynomial of degree at most $k-l+1$. For $z < X_0(n-l+1)$, using (\ref{relation}), we have \[\frac{1}{\rho}h^n_k(l-1,z) = \sum_{j=0}^{X_0(n-l+1)-z-1}(1-\rho)^{X_0(n-l+1)-z-j-1}h^n_k(l,X_0(n-l+1)-j)\] for the same reasoning as above.
Now we can verify that the function $\Phi^n_{\rho,k}(z)$ satisfies the two conditions in (\ref{TASEPformula}). Using (\ref{psiformula}), 
\begin{equation}
\begin{split}
\sum_{z}\Psi^n_{\rho,l}(z)\Phi^n_{\rho,k}(z) &=  \sum_{z_1,z_2}\sum_{z}R_{\rho,t}(z,z_1)Q^{-l}_{\rho}(z_1,X_0(n-l))\frac{1-\rho}{\rho}h^n_k(0,z_2)R_{\rho,t}^{-1}(z_2,z)\\
& = \sum_{z \in \mathbb{Z}}\frac{1-\rho}{\rho}Q^{-l}_{\rho}(z,X_0(n-l))h^n_k(0,z)  = \frac{1-\rho}{\rho}(Q^*_\rho)^{-l}h^n_k(0,X_0(n-l)).
\end{split}
\end{equation}
For $l\leq k$, using the relations (\ref{eq1}),(\ref{eq2}), we have 
\begin{equation}
\frac{1-\rho}{\rho}(Q^*_\rho)^{-l}h^n_k(0.X_0(n-l)) = \frac{1-\rho}{\rho}h^n_k(l,X_0(n-l))=\mathbbm{1}_{k=l}.
\end{equation}
For $l>k$, using the fact that $\frac{1}{(1-\rho)^z}\in \ker(Q^*)^{-1}$:
\begin{equation}
(Q^*)^{-l}h^n_k(0,X_0(n-l)) = (Q^*)^{-(l-k-1)}(Q^*)^{-1}h^n_k(k,X_0(n-l))=0.
\end{equation}
Now we check $(1-\rho)^x\Phi^n_k(x)$ is a polynomial of degree at most $n-1$ in $x$ for each $k$. Using the explicit formula of $R_{\rho,t}^{-1}$, we have 
\begin{equation}
(1-\rho)^x\Phi^n_k(x) = \sum_y\frac{1-\rho}{\rho}h^n_k(0,y+x)e^{t}\frac{(1-\rho)^{y+x}t^{y}}{y!}\mathbbm{1}_{y\geq 0}.
\end{equation}
Using the fact that $(1-\rho)^{y+x}h^n_k(0,y+x)$ is a polynomial in $y+x$ with degree at most $k$, the sum in $y$ is absolutely convergent, thus the left-hand side is a polynomial in $x$ with degree at most $k$.
\end{proof}
\begin{remark}
Note that the dependence of $h^n_k(l,z)$ on the positions of the particles $X_0(1)>X_0(2)>\cdots>X_0(n)$ is not explicit in the notation. This dependence will be important in the following sections; thus we use the notation $h^{n,X_0}_k$to denote its dependence on the initial condition. Whenever the dependence causes no confusion, we will omit it to avoid heavy notation.
\end{remark}
When $z \leq X_0(n-l)$, the function $h^N_{k}(l,z)$ has a probabilistic meaning. 
\begin{lemma}
Define stopping times $$\tau^{l,n}=\min\{m\in\{l,\cdots,n-1\}: B^*_m > X_0(n-m)\}$$ with the convention that $\min \emptyset = \infty$. Then \[h^n_k(l,z) = \mathbb{P}_{B^*_{l-1}=z}(\tau^{l,n}=k).\]
\end{lemma}
\begin{proof}
When $l = k$, $\mathbb{P}_{B^*_{k-1}=z}(\tau^{k,n}=k)$ is the probability that the random walk starts from $z$ and goes above $X_0(n-k)$ in the next jump, so 
\begin{equation}
\begin{split}
\mathbb{P}_{B^*_{k-1}=z}(\tau^{k,n}=k) &= \sum_{j = 1}^\infty \mathbb{P}(B^*_k = X_0(n-k)+j)\\
&=\sum_{j=1}^\infty (1-\rho)^{X_0(n-k)+j-z-1}\rho = \rho(1-\rho)^{X_0(n-k)-z-1} = h^n_k(k,z).
\end{split}
\end{equation}
If $l<k, \mathbb{P}_{B^*_{l-1}=X_0(n-l)}(\tau^{l,n}=k) = 0$ since the random walk starts at $X_0(n-l)$, then it will definitely stay above $X_0(n-l)$ in the first jump. Lastly, we verify condition (\ref{eq1}). Notice 
\begin{equation}
\begin{split}
\mathbb{P}_{B^*_{l-1}=z-1}(\tau^{l,n}=k)& = \sum_{j = 1}^\infty\mathbb{P}_{B^*_{l}=z+j-1}(\tau^{l+1,n}=k)P(B^*_l-B^*_{l-1} = j)\\
&=\sum_{j = 1}^\infty\mathbb{P}_{B^*_{l}=z+j-1}(\tau^{l+1,n}=k)\rho(1-\rho)^{j-1}\\
&= \mathbb{P}_{B^*_l=z}(\tau^{l+1,n}=k)\rho+(1-\rho)\sum_{j = 1}^\infty\mathbb{P}_{B^*_{l}=z+j}(\tau^{l+1,n}=k)\rho(1-\rho)^{j-1}\\
&=\mathbb{P}_{B^*_l=z}(\tau^{l+1,n}=k)\rho+(1-\rho)\mathbb{P}_{B^*_{l-1}=z}(\tau^{l,n}=k).\\
\end{split}
\end{equation}
which is exactly the relation (\ref{eq1}). By the uniqueness of equations (\ref{eq1})-(\ref{eq3}), we proved the lemma.
\end{proof}
Now we define 
\begin{equation}
G_{0,n}(z_1,z_2)=\sum_{k=0}^{n-1}\frac{1-\rho}{\rho}Q^{n-k}(z_1,X_0(n-k))h^n_k(0,z_2).
\end{equation}
Using Theorem (\ref{hnkdef}) and equation (\ref{psiformula}), we can write our kernel as:
\begin{equation}\label{kernelGform}
\sum_{k=0}^{n-1}\Psi^n_{\rho,k}(x_1)\Phi^n_{\rho,k}(x_2) = R_{\rho,t}Q^{-n}_{\rho}G_{0,n}R^{-1}_{t,\rho}(x_1,x_2).
\end{equation}
When $z_2\leq X_0(n)$, $G_{0,n}$ is the probability for the walk starting at $z_2$ at time $-1$, to end up at $z_1$ after $n$ steps, having hit the curve $(X_0(n-m))_{m=0,\cdots,n-1}$ in between. To see it, using the memoryless property of the geometric distribution, we have for all $z \leq X_0(n)$,
\begin{equation}
\mathbb{P}_{B^*_{-1}=z}(\tau^{0,n} = k,B^*_k=y) = \rho(1-\rho)^{y-X_0(n-k)-1}\mathbb{P}_{B_{-1}^*=z}(\tau^{0,n}=k).
\end{equation}
Then 
\begin{equation}\label{probabilityinterpretforG}
\begin{split}
G_{0,n}(z_1,z_2) &= \sum_{k=0}^{n-1}\mathbb{P}_{B^*_{-1}=z_2}(\tau^{0,n}= k)(Q^*)^{n-k}(X_0(n-k),z_1)\\
&=\sum_{k=0}^{n-1}\sum_{z>X_0(n-k)}\mathbb{P}_{B^*_{-1}=z_2}(\tau^{0,n}=k,B^*_k=z)(Q^*)^{n-k-1}(z,z_1)\\
&=\mathbb{P}_{B^*_{-1}=z_2}(\tau^{0,n}<n,B^*_{n-1}=z_1).
\end{split}
\end{equation}Now we want to extend the expression    (\ref{probabilityinterpretforG}) to all $z_2\in \mathbb{Z}$. We first have an extension for the kernel $Q^n(x,y)$. For each fixed $x$ and $n\geq1$, $(1-\rho)^{y}Q^n(x,y)$ extends in $y$ to a polynomial $(1-\rho)^{-y}\overline{Q}(x,y)$ of degree $n-1$ where
\begin{equation}
\overline{Q}^{(n)}(x,y)=\frac{1}{2\pi i}\int \frac{(1+v)^{x-y-1}}{\rho^n(1-\rho)^{x-y-n}v^n}dv.
\end{equation}
Notice that 
\begin{equation}
\overline{Q}^{(n)}(x,y) = Q^n(x,y) \text{ for } x-y\geq 1.
\end{equation}
However, $\overline{Q}^{(n)}$ behaves no longer like a semigroup, since $\overline{Q}^{(n)}\overline{Q}^{(m)}$ is divergent. Moreover, $Q\overline{Q}^{(n)}, \overline{Q}^{(n)}Q$ are also divergent. 
Now we define the hitting time of a geometric random walk starting from right to left, which is 
\begin{equation}
\tau = \min \{m\geq 0: B_m>X_0(m+1)\}
\end{equation}
where $B_m$ is random walk with Geom($\rho$) in the negative direction. Now we can give the extension for $G$:
\begin{lemma}
For all $z_1,z_2 \in \mathbb{Z}$, we have 
\begin{equation}
G_{0,n}(z_1,z_2) = \mathbb{E}_{B_0=z_1}[\overline{Q}^{(n-\tau)}(B_\tau,z_2)1_{\tau<n}]
\end{equation}
\end{lemma}
Using the probabilistic meaning of the operator $G_{0,n}, Q^*$ and $h^k_n(0,z)$, we have the following splitting formula for $h^n_k(0,z)$, which is not present in \cite{MQR2021}.
\begin{lemma}\label{splittinglemma}
Assume $n = N_1+N_2$ for some positive integers $N_1,N_2$. For $0 \leq k\leq N_1$, we have
\begin{equation}
        h_{k+N_2}^{N_1+N_2}(0,z)=\sum_{z_3}(Q^{N_2}(z_3,z)-G_{0,N_2}(z_3,z))h_k^{N_1}(0,z_3).
\end{equation}
\end{lemma}
\begin{proof}
For $z\leq X_0(N_1+N_2)$, using the probabilistic meaning of the function,
\begin{equation}
    \begin{split}
        &h_{k+N_2}^{N_1+N_2}(0,z)=\mathbb{P}_{B^*(-1)=z}(\tau^{0,N_1+N_2}=k+N_2)\\
            &=\sum_{z_3\leq X_0(N_1)}\mathbb{P}_{B^*(-1)=z}(\text{no hit in } N_2 \text{ steps}, \text{ end at } z_3)\mathbb{P}_{B*(N_1+1)=z_3}(\tau^{N_2,N_1+N_2}=k+N_2)\\
            &=\sum_{z_3\leq X_0(N_1)}(\mathbb{P}_{B^*(-1)=z}(\text{end at }z_3)-\mathbb{P}_{B^*(-1)=z}(\text{hit in first } N_2 \text{ steps, end at }z_3))h_k^{N_1}(0,z_3)\\
            &=\sum_{z_3\leq X_0(N_1)}((Q^*)^{N_2}(z,z_3)-G_{0,N_2}(z_3,z))h_k^{N_1}(0,z_3).\\
          \end{split}
\end{equation}
The sum is over $z_3 \leq X_0(N_2)$ because we want $h^{N_1}_k(0,z_3)= \mathbb{P}_{B*(N_2-1)=z_3}(\tau^{N_2,N_1+N_2}=k+N_2) $. But we can remove this constraint since when $z_3 > X_0(N_2)$, $(Q^*)^{N_2}(z,z_3)-G_{0,N_2}(z_3,z)=0$.

Now we want to prove this equality when $z >X_0(N_1+N_2)$. First, notice that by definition, $h^{N_1}_k(0,z) = h^{N_1+N_2}_{k+N_2}(N_2,z)$, since they solve the equations in (\ref{hnkdef}) with the same boundary condition. Recall from the proof of (\ref{hnkdef}) that for $z\geq X_0(n-l+1)$,
\begin{equation}
h^n_k(l-1,z) = \sum_{y=X_0(n-l+1)+1}^z-h^n_k(l,y)\overline{Q}(y,z),
\end{equation}and for $z<X_0(n-l+1)$, 
\begin{equation}
h^n_k(l-1,z) = \sum_{y=z+1}^{X_0(n-l+1)}h^n_k(l,y)Q(y,z).
\end{equation} We will keep using these two formulas to solve $h^{N_1+N_2}_{k+N_2}(0,z)$ from $h^{N_1+N_2}_{k+N_2}(N_2,z)$. Now we prove by induction.
Assume $N_2 =1$,
\begin{equation}
\begin{split}
h^{N_1+1}_{k+1}(0,z)= \sum_{y=X_0(n)+1}^z-h^{N_1+1}_{k+1}(1,y)\overline{Q}(y,z).
\end{split}
\end{equation}
Notice that when $z\geq y\geq X_0(n)+1$, we have $\overline{Q}(y,z) = G_{0,1}(y,z),Q(y,z)=0$, so we can write
\begin{equation}
 h^{N_1+1}_{k+1}(0,z)= \sum_{y=X_0(n)+1}^zh^{N_1+1}_{k+1}(1,y)(Q(y,z)-G_{0,1}(y,z)).
 \end{equation} 
 Then notice that $y>z$, $Q(y,z)= G_{0,1}(y,z)$, and when $y\leq X_0(n)$, both $Q(y,z)$and $G_{0,1}(y,z)$ are 0, thus we can get rid of the constrain on $y$.
 Now assume 
 \begin{equation}
  h_{k+N_2}^{N_1+N_2}(0,z)=\sum_{z_3=X_0(n)+N_2}^{z}(Q^{N_2}(z_3,z)-G_{0,N_2}(z_3,z))h_{k+N_2}^{N_1+N_2}(N_2,z_3).
 \end{equation}
 We prove the same formula holds for $N_2$ replaced by $N_2+1$.\\
By the assumption,
\begin{equation}
h^{N_1+N_2+1}_{k+N_2+1}(0,z)= \sum_{z_3=X_0(n+1)+N_2}^{z}h_{k+N_2+1}^{N_1+N_2+1}(N_2,z_3)(Q^{N_2}(z_3,z)-G_{0,N_2}(z_3,z)).
\end{equation}
Depending on whether $z_3> X_0(N_1+1)$ or $z_3\leq X_0(N_1+1)$, we have different formulas for $h_{k+N_2+1}^{N_1+N_2+1}(N_2,z_3)$ in terms of $h_{k+N_2+1}^{N_1+N_2+1}(N_2+1,z_3)$.
 If $z_3\leq X_0(N_1+1)$, 
 \begin{equation}
h_{k+N_2+1}^{N_1+N_2+1}(N_2,z_3) = \sum_{y=z_3+1}^{X_0(N_1+1)}h_{k+N_2+1}^{N_1+N_2+1}(N_2+1,y)Q(y,z_3).
 \end{equation}
 If $z_3\geq X_0(N_1+1)$, 
  \begin{equation}
 h_{k+N_2+1}^{N_1+N_2+1}(N_2,z_3) = -\sum_{y=X_0(N_1+1)+1}^{z_3}h_{k+N_2+1}^{N_1+N_2+1}(N_2+1,y)\overline{Q}(y,z_3).
  \end{equation}
  So,
  \begin{equation}
  \begin{split}
&\sum_{z_3=X_0(n+1)+N_2}^{X_0(N_1+1)}\sum_{y=z_3+1}^{X_0(N_1+1)}h_{k+N_2+1}^{N_1+N_2+1}(N_2+1,y)Q(y,z_3)(-G_{N_1+1,n+1}(z_3,z))\\
&=\sum_{y=X_0(n+1)+N_2+1}^{X_0(N_1+1)}\sum_{z_3=X_0(n+1)+N_2}^{X_0(N_1+1)\wedge (y-1)}h_{k+N_2+1}^{N_1+N_2+1}(N_2+1,y)Q(y,z_3)(-G_{N_1+1,n+1}(z_3,z))\\
&=\sum_{y=X_0(n+1)+N_2+1}^{X_0(N_1+1)}h_{k+N_2+1}^{N_1+N_2+1}(N_2+1,y)(-G_{N_1,n+1}(y,z)).
  \end{split}
  \end{equation}
 The last equality is true because of the range of $y$. If we start the walk below $X_0(N_1+1)$, then it cannot hit $X_0(N_1+1)$.
  Now look at the second part,
  \begin{equation}
  \begin{split}
 &\sum_{z_3=X_0(N_1+1)+1}^z \sum_{y=X_0(N_1+1)+1}^{z_3}-h_{k+N_2+1}^{N_1+N_2+1}(N_2+1,y)\overline{Q}(y,z_3)(-G_{N_1+1,n+1}(z_3,z))\\
 &=\sum_{y=X_0(N_1+1)+1}^{z}\sum_{z_3=y}^z h_{k+N_2+1}^{N_1+N_2+1}(N_2+1,y)(-\overline{Q}(y,z_3))(-G_{N_1+1,n+1}(z_3,z))\\
 &=\sum_{y=X_0(N_1+1)+1}^{z}h_{k+N_2+1}^{N_1+N_2+1}(N_2+1,y)(-G_{N_1,n+1}(y,z))\\
  \end{split}
  \end{equation}
  This is because combining these two equations, we get 
\begin{equation}
    \begin{split}
        &h^{N_1+N_2+1}_{k+N_2+1}(0,z) = \sum_{y=X_0(n+1)+N_2+1}^z h_{k+N_2+1}^{N_1+N_2+1}(N_2+1,y)(-G_{N_1,n+1}(y,z))
    \end{split}
\end{equation}
Notice $Q^{N_2+1}(y,z) = 0$ for $y\leq z+N_2$, so we can add the $Q^{N_2+1}$ term:
\begin{equation}
    h^{N_1+N_2+1}_{k+N_2+1}(0,z) = \sum_{y=X_0(n+1)+N_2+1}^z h_{k+N_2+1}^{N_1+N_2+1}(N_2+1,y)(Q^{N_2+1}(y,z)-G_{N_1,n+1}(y,z))
\end{equation}
Then the summation can be extended for all $y$, since when $y< X_0(n+1)+N_2+1$, both $Q^{N_2+1}(y,z), G_{N_1,n+1}(y,z)$ are 0, and when $y>z, Q^{N_2+1}(y,z)= G_{N_1,n+1}(y,z)$
\end{proof}
For the further purpose of taking scaling limit, we want to define the following two operators, it is just a rename of combination of terms in equation (\ref{kernelGform}):
\begin{equation}\label{formulaSandSbar}
\begin{split}
S_{-t,-n}(z_1,z_2) &=  (e^{-(1-\rho)t\nabla^-}Q^{-n})(z_1,z_2) = \frac{1}{2\pi i} \int_{\Gamma_0}dw(\frac{1-\rho}{\rho})^{n}\frac{(1-\rho)^{z_2-z_1}(1-w)^{n}}{w^{n+1+z_2-z_1}}e^{t(w-(1-\rho))}\\
\overline{S}_{-t,n}(z_1,z_2) &= (\overline{Q}^{(n)}e^{(1-\rho)t\nabla^-})(z_1,z_2)= \frac{1}{2\pi i}\int_{\Gamma_0}dw (\frac{\rho}{1-\rho})^{n-1}\frac{(1-w)^{z_2-z_1+n-1}}{(1-\rho)^{z_2-z_1}w^n}e^{t(w-\rho)}.\\
\end{split}
\end{equation}
$\Gamma_0$ is a simple counterclockwise loop around $0$ that does not enclose $1$. Now we define 
\begin{equation}
\overline{S}^{\mathrm{epi}(X_0,m)}_{-t,n}(z_1,z_2) = \mathbb{E}_{B_0=z_1}[\overline{S}_{-t,n-\tau}(B_\tau,z_2)1_{\tau<m}].
\end{equation}
Here $m<n$ and if $m=n$ we will omit the superscript $m$.
Now we see that our TASEP kernel can be written as:
\begin{equation}
 \sum_{k=0}^{n-1}\Psi^n_k(x)\Phi^n_k(y) = (S^*_{-t,-n}\overline{S}^{\mathrm{epi}(X_0)}_{-t,n})(x,y)
\end{equation}

\section{Splitting of the kernel and asymptotic independence}
We start with the following initial condition: Bernoulli $\rho_-$ on the left of the origin and Bernoulli $\rho_+$ on the right of the origin.  Furthermore, we assume that there are $n_r$ particles on the right side of the origin.  So we have $$X_0(n_r+1)<0\leq X_0(n_r)<X_0(n_r-1)<\cdots<X_0(1).$$  $n_r $ will be large enough so that it does not affect the fluctuation around the shock. For any particle $X_0(n)$ on the left of the origin, we always write $n = n_l+n_r$ for some $n_l>0$. Let us also introduce some extra notation. Let $X^l_0(1) = X_0(n_r+1), X^l_0(2)= X_0(n_r+2),\cdots$, which should be thought of as a new initial condition which only has the left piece; similarly the right piece is: $X^r_0(1) = X_0(1),\cdots, X^r_0(n_r) = X_0(n_r)$. Denote the initial condition as $\eta_0(x)$, and its corresponding height function as $h_0(x)$ with $h_0(0) = 0$.

By Theorem \eqref{TASEPformula},
\begin{equation}
\mathbb{P}(X_t(n)>a)=\int \mu(X_0) \det(I-\overline{1}^a K_{X_0} \overline{1}^a)_{l^2(\mathbb{Z})}.
\end{equation}
where
\begin{equation}
K_{{X_0}}(x_1,x_2) = \sum_{k=0}^{n-1}\Psi^n_k(x_1)\Phi^n_k(x_2) = (\sum_{k=0}^{n_l-1}+\sum_{k=n_l}^n)\Psi^n_k(x_1)\Phi^n_k(x_2):= K^l(x_1,x_2)+K^r(x_1,x_2).
\end{equation}

Define the first summation to be $K^l$, and the second to be $K^r$, which encodes the information from the left and right pieces. From the bi-orthogonalization relation \eqref{eqn:biorthogonalizationrelation}, 
\begin{equation}
K^rK^l = 0.
\end{equation}
Thus we can write the determinant as:
\begin{equation}
\begin{split}
    \mathbb{P}(X_t(n)> a)& = \mathbb{E}_{X_0}\big[\det(I-\overline{1}^a(K^l+K^r)\overline{1}^a)_{l^2(\mathbb{Z})}\big]\\
    &= \mathbb{E}_{X_0}\big[\det((I-\overline{1}^aK^r\overline{1}^a)(I-\overline{1}^aK^l\overline{1}^a)+\overline{1}^aK^r 1_aK^l\overline{1}^a)_{l^2(\mathbb{Z})}\big].
\end{split}
\end{equation}
    We want to consider the following scaling limit:  
    \begin{equation}\label{scaling}
    \begin{split}
    &n_\varepsilon = \rho_+\rho_-(\varepsilon^{-3/2}\bt+\varepsilon^{-1}\bx)+\frac{\rho_-N_r+\rho_+N_l}{\rho_+-\rho_-}+n_{r,\varepsilon}\\
    &a_\varepsilon = (1-\rho_--\rho_+)(\varepsilon^{-3/2}\bt+\varepsilon^{-1}\bx)-\frac{N_r+N_l}{\rho_+-\rho_-}-\bs\varepsilon^{-1/2},\quad t_\varepsilon = (\varepsilon^{-3/2}\bt+\varepsilon^{-1}\bx).
    \end{split}
    \end{equation}
 We define the multiplicative operator, $$M_{\rho} f(x)  = \rho^xf(x).$$Under this scaling, we have the following theorem.
    \begin{theorem}\label{abconvergence}
    \begin{equation}
    \begin{split}
        &\lim_{\varepsilon\to 0}\mathbb{E}_{X_0}\det(I-\overline{1}^{a_\varepsilon}M_{1-\rho_-}K_\varepsilon^lM_{1/(1-\rho_-)}\overline{1}^{a_\varepsilon})_{l^2(\mathbb{Z})} = \mathbb{E}_{\mathfrak{h}^l}F^{\mathfrak{h}^l}(s),\\
        &\lim_{\varepsilon\to 0}\mathbb{E}_{X_0}\det(I-\overline{1}^{a_\varepsilon}M_{1-\rho_+}K_\varepsilon^rM_{1/(1-\rho_+)}\overline{1}^{a_\varepsilon})_{l^2(\mathbb{Z})} = \mathbb{E}_{\mathfrak{h}^r}F^{\mathfrak{h}^r}(s).\\
    \end{split}
    \end{equation}
    where $\mathfrak{h}^l,\mathfrak{h}^r$ are independent two-sided Brownian motions. $F^{\mathfrak{h}}$ is a distribution function indexed by functions in support of the Wiener measure. Its exact form will be given at the end of Section \eqref{subsec:convergencekl},\eqref{subsec:convergencekr}.\end{theorem}
   We will prove the theorem in the Appendix. Together with the proof of Theorem (\ref{abconvergence}), we also prove the uniform bound on the trace norm of the kernel $\overline{1}^{a_\varepsilon}M_{1-\rho_-}K_\varepsilon^lM_{1/(1-\rho_-)}\overline{1}^{a_\varepsilon}$ and $\overline{1}^{a_\varepsilon}M_{1-\rho_+}K_\varepsilon^rM_{1/(1-\rho_+)}\overline{1}^{a_\varepsilon}$. Then we can prove the next theorem, which shows that the two pieces of the initial configuration are asymptotically independent. We drop $\varepsilon$ super/sub scripts in the following theorem and proof since there will be no confusion that we are considering the scaling limit.
   \begin{theorem}
           \begin{multline}
               \lim_{\varepsilon \to 0}\mathbb{E}_{X_0}\det\big((I-\overline{1}^aK^r\overline{1}^a)(I-\overline{1}^aK^l\overline{1}^a)+\overline{1}^aK^r1_aK^l\overline{1}^a\big) \\= \lim_{\varepsilon \to 0}\mathbb{E}_{X_0}\det\big((I-\overline{1}^aK^r\overline{1}^a)(I-\overline{1}^aK^l\overline{1}^a)\big).
           \end{multline}
   \end{theorem}
   \begin{proof}
   Let 
   $Y = (I-\overline{1}^aK^r\overline{1}^a)\overline{1}^a(I-\overline{1}^aK^l\overline{1}^a),\quad E = \overline{1}^aK^r1_aK^l\overline{1}^a.$
   Then 
   \begin{multline}
       \lim_{\varepsilon \to 0}\mathbb{E}_{X_0}\det\big((I-\overline{1}^aK^r\overline{1}^a)(I-\overline{1}^aK^l\overline{1}^a)+\overline{1}^aK^r1_aK^l\overline{1}^a\big)\\=\lim_{\varepsilon \to 0}\mathbb{E}_{X_0}\big(\det(Y)\det(I+Y^{-1}E)\big).
   \end{multline}
   We try to estimate the trace norm of $Y^{-1}E$. Since we know that $\overline{1}^aM_{1-\rho_-}K^lM_{1/(1-\rho_-)}\overline{1}^a$ and $\overline{1}^aM_{1-\rho_+}K^rM_{1/(1-\rho_+)}\overline{1}^a$ are uniformly bounded in the trace norm, we want to write \[K^r= M_{1/(1-\rho_+)}\mathcal{K}^rM_{(1-\rho_+)},\quad K^l= M_{1/(1-\rho_-)}\mathcal{K}^lM_{(1-\rho_-)}\] where $\overline{1}^a\mathcal{K}^l\overline{1}^a, \overline{1}^a\mathcal{K}^r\overline{1}^a$ are uniformly bounded in the trace norm. Now we have
    \begin{equation}
       \begin{split}
            Y^{-1}E  &=  M_{1/(1-\rho_-)}(I-\overline{1}^a\mathcal{K}^l\overline{1}^a)^{-1}M_{(1-\rho_-)}M_{1/(1-\rho_+)}(I-\overline{1}^a\mathcal{K}^r\overline{1}^a)^{-1}M_{(1-\rho_+)}\\
           &\cdot   M_{1/(1-\rho_+)}\overline{1}^a\mathcal{K}^r M_{(1-\rho_+)}1_a M_{1/(1-\rho_-)}\mathcal{K}^l\overline{1}^a M_{(1-\rho_-)} \\
           \end{split}
           \end{equation}
    Conjugating the kernel by $M_{1/(1-\rho_-)}$ and noticing that the projection operator $\overline{1}^a$ commutes with the multiplication operator $M_\rho$, thus we have:
    \begin{equation}\label{eqn:65}
    \begin{split}
    &\lVert M_{1-\rho_-}Y^{-1}EM_{1/(1-\rho_-)} \rVert_{tr} = \lVert (I-\overline{1}^a\mathcal{K}^l\overline{1}^a)^{-1}M_{(1-\rho_-)}M_{1/(1-\rho_+)}(I-\overline{1}^a\mathcal{K}^r\overline{1}^a)^{-1}\\
           &\cdot   \overline{1}^a\mathcal{K}^r M_{(1-\rho_+)}1_a M_{1/(1-\rho_-)}\mathcal{K}^l\overline{1}^a \rVert_{tr}\\
    \end{split}
    \end{equation}
    Due to the fact that $(I-\overline{1}^a\mathcal{K}^r\overline{1}^a)^{-1}\overline{1}^a= \overline{1}^a(I-\overline{1}^a\mathcal{K}^r\overline{1}^a)^{-1}\overline{1}^a$,
    \begin{equation}\label{eqn:66}
        \eqref{eqn:65} = \lVert (I-\overline{1}^a\mathcal{K}^l\overline{1}^a)^{-1}M_{(1-\rho_-)}\overline{1}^aM_{1/(1-\rho_+)}(I-\overline{1}^a\mathcal{K}^r\overline{1}^a)^{-1}\cdot   \overline{1}^a\mathcal{K}^r 1_a M_{\tfrac{1-\rho_+}{1-\rho_-}}\mathcal{K}^l\overline{1}^a  \rVert_{tr}
    \end{equation}
    where the only change is that we add $\overline{1}^a$ between $M_{1-\rho_-}$ and $M_{1/(1-\rho_+)}$. The estimate the trace norm of \eqref{eqn:66} by
    \begin{equation}\label{eqn:67}
      \begin{split}
           \eqref{eqn:66} \leq \lVert(I-\overline{1}^a\mathcal{K}^l\overline{1}^a)^{-1} \rVert_{op}\lVert(I-\overline{1}^a\mathcal{K}^r\overline{1}^a)^{-1} \rVert_{op}\lVert M_{(1-\rho_-)/(1-\rho_+)}\overline{1}^a\rVert_{tr}\\
           \lVert    \overline{1}^a\mathcal{K}^r 1_a M_{\sqrt{\tfrac{1-\rho_+}{1-\rho_-}}}\rVert_{op}\lVert M_{\sqrt{\tfrac{1-\rho_+}{1-\rho_-}}}1_a \mathcal{K}^l\overline{1}^a  \rVert_{op}
      \end{split}
    \end{equation}
 For the first two terms, since $\lim_{\varepsilon\to 0}\det(I-\overline{1}^a\mathcal{K}^l\overline{1}^a)$ and $\lim_{\varepsilon\to 0}\det(I-\overline{1}^a\mathcal{K}^r\overline{1}^a)$ are distribution functions that are strictly greater than 0. From Fredholm theory, for $\varepsilon$ small enough, the kernels are invertible and there exists $\delta>0$ such that their operator norms are greater than $\delta$. Thus, the first two terms are uniformly bounded. For the last two terms, notice that $\sqrt{\tfrac{1-\rho_+}{1-\rho_-}}<1$. Using the proof that $\overline{1}^a\mathcal{K}^l\overline{1}^a,\overline{1}^a\mathcal{K}^r\overline{1}^a$ are bounded in trace norm, we can easily see that 
 \[\lVert    \overline{1}^a\mathcal{K}^r 1_a M_{\sqrt{\tfrac{1-\rho_+}{1-\rho_-}}}\rVert_{op}\leq (\tfrac{1-\rho_+}{1-\rho_-})^{a/2} \lVert    \overline{1}^a\mathcal{K}^r \overline{1}^a \rVert_{tr},\quad \lVert M_{\sqrt{\tfrac{1-\rho_+}{1-\rho_-}}}1_a \mathcal{K}^l\overline{1}^a  \rVert_{op}\leq (\tfrac{1-\rho_+}{1-\rho_-})^{a/2} \lVert    \overline{1}^a\mathcal{K}^l \overline{1}^a \rVert_{tr}\]
 Thus, there exist a constant $C$ such that the product of four operator norms in \eqref{eqn:67} is less than: $C(\tfrac{(1-\rho_+}{1-\rho_-})^a$. Now we compute the $\lVert M_{(1-\rho_-)/(1-\rho_+)}\overline{1}^a\rVert_{tr}$. 
 \begin{equation}
 \rVert M_{\frac{1-\rho_-}{1-\rho_+}}\overline{1}^a\rVert_{tr}\leq \sum_{k\leq a}(\frac{1-\rho_-}{1-\rho_+})^k=(\frac{1-\rho_-}{1-\rho_+})^a(\frac{1-\rho_-}{\rho_+-\rho_-})\varepsilon^{1/2}
 \end{equation}
 The last equality is due to that here $M_\rho$ is act on space $\varepsilon^{1/2}\mathbb{Z}$, and $a = (1-\rho_--\rho_+)\varepsilon^{-3/2}-\frac{N_r+N_l}{\rho_+-\rho_-}-\bs\varepsilon^{-1/2}$.
 so putting them all together we can see the norm goes to 0 as $\varepsilon \to 0$.
   \end{proof}

\section{\texorpdfstring{Proof of operator $K^l$ and $K^r$ converges}{Proof of operator Kl and Kr converges}}
In this section, we study the limit of 
\begin{equation}
    \begin{split}
        &\lim_{\varepsilon\to 0}\mathbb{E}_{X_0}\big[\det\big((I-\overline{1}^aK^l\overline{1}^a)(I-\overline{1}^aK^r\overline{1}^a)\big)\big]\\
        &=\lim_{\varepsilon\to 0}\mathbb{E}_{X_0}\big[\det(I-\overline{1}^aM_{1-\rho_-}K^lM_{1/(1-\rho_-)}\overline{1}^a)_{l^2(\mathbb{Z})}\det(I-\overline{1}^aM_{1-\rho_+}K^rM_{1/(1-\rho_+)}\overline{1}^a)_{l^2(\mathbb{Z})}\big].\\
    \end{split}
\end{equation}
Recall that we are considering the following scaling of the problem.
\begin{equation}
    \begin{split}
    &n = n_{r}+n_{l}, \quad n_{l} = \rho_-\rho_+(\varepsilon^{-3/2}\bt+\varepsilon^{-1}\bx)+n_s-\bs\varepsilon^{-1/2}\\
    &a = (1-\rho_--\rho_+)(\varepsilon^{-3/2}\bt+\varepsilon^{-1}\bx)-\frac{N_r+N_l}{\rho_+-\rho_-},\quad t = \varepsilon^{-3/2}\bt+\varepsilon^{-1}\bx.
    \end{split}
    \end{equation} 
    Recall $n_s = \frac{\rho_-N_r+\rho_+N_l}{\rho_+-\rho_-}$, which is the shift term that cancels the effect of the initial randomness from two sides. Also recall that we define the notation
    \[\tilde{\bx}_-=(\rho_--\rho_+)2^{-1}\chi_-^{-1/3}\bx,\quad \tilde{\bx}_+=(\rho_+-\rho_-)2^{-1}\chi_+^{-1/3}\bx.\]
    \[\tilde{\bs}_- = \chi_-^{-2/3}(1-2\rho_-)\bs, \quad \tilde{\bs}_+ = \chi_+^{-2/3}(1-2\rho_+)\bs.\]
    Now we want to further split the left piece into $$n_{ll} =  \rho_-^2\varepsilon^{-3/2}\bt+\rho_-\rho_+\varepsilon^{-1}\bx-\bs\varepsilon^{-1/2},\quad n_{lr} = n_{l}-n_{ll}.$$ We split the right piece into $$n_{rl} = \rho_+(\rho_+-\rho_-)\varepsilon^{-3/2}\bt,\quad n_{rr} = n_r-n_{rl}$$ 
We use notation $X^{ll}_0,X^{lr}_0,X^{rl}_0,X^{rr}_0$ for the corresponding piece of the initial condition.
We also define $n_{-} = n_{ll} = \rho_-^2\varepsilon^{-3/2}\bt+\rho_-\rho_+\varepsilon^{-1}\bx+n_s-\bs\varepsilon^{-1/2}, n_+ =n_l+n_{rl} =  \rho_+^2\varepsilon^{-3/2}\bt+\rho_-\rho_+\varepsilon^{-1}\bx+n_s-\bs\varepsilon^{-1/2}$, these two numbers are essential for the scaling.

Before we state the convergence results, let us recall some operator from \cite{MQR2021}. For $\bt>0$, 
\begin{equation}
    \begin{split}
        \bS_{\bt,\bx}(\bu,\bv):&=\tfrac{1}{2\pi i}\int_{C^{\pi/3}_{1}} dw e^{\bt w^3/3+\bx w^2+(\bu-\bv)w}\\
    &= \bt^{-1/3}e^{\tfrac{2\bx^3}{3\bt^2}-\tfrac{(\bu-\bv)\bx}{t}}\Ai(-\bt^{-1/3}(\bu-\bv)+\bt^{-4/3}\bx^2).
    \end{split}
\end{equation}
where $C_a^{\pi/3} = \{a+re^{i\pi/3}:r\in [0,\infty)\}\cup  \{a+re^{-i\pi/3}:r\in [0,\infty)\}$. This is the integral kernel for the operator $e^{\bx\partial^2+\bt\partial^3/3}$. For $\bt=0$, the operator is still well defined for $\bx>0$. For $\bt_1,\bt_2>0$, it behaves like a group, i.e. $\bS_{\bt_1,\bx_1}\bS_{\bt_2,\bx_2}=\bS_{\bt_1+\bt_2,\bx_1+\bx_2}$. One useful property of the operator is: $\bS_{-\bt,\bx}= (\bS_{\bt,\bx})^*$. Furthermore, we define
\begin{equation}
    \bS^{\mathrm{epi}(\mathfrak{g})}_{\bt,\bx}(\bv,\bu)=\mathbb{E}_{\bB(0)=\bv}[\bS_{\bt,\bx-\btau}(\bB(\btau),u)1_{\btau<\infty}].
\end{equation}
$\btau$ is the first time that a Brownian motion with variance $2$ starting from $\bv$ hits the epigraph of $\mathfrak{g}$, where $\mathrm{epi}(\mathfrak{g}) = \{(m,y):y >\mathfrak{g}(m)\}$. Notice that $\bS^{\mathrm{epi}(\mathfrak{g})}_{\bt,\bx}(\bv,\bu)$ can also be written as 
\begin{equation}\label{eqn:sepiseconddef}
    \bS^{\mathrm{epi}(\mathfrak{g})}_{\bt,\bx} = \lim_{\bT\to\infty}\bS^{\mathrm{epi}(\mathfrak{g})}_{[0,\bT]}\bS_{\bx-\bT}
\end{equation}
where $\bS^{\mathrm{epi}(\mathfrak{g})}_{[0,\bT]}(\bv,\bu)$ is the transition density of $\bB$ to go from $\bv$ at time $0$ to $\bu$ at time $\bT$ hitting the epigraph of $\mathfrak{g}$ in $[0,\bT]$.

The proof involves the convergence of multiple parts of the kernel. To make the main structure clear, all the proofs of the convergence results are given in Section (5).
\subsection{\texorpdfstring{Convergence of $K^l$}{Convergence of Kl}}\label{subsec:convergencekl}
Now look at kernel $K^l(z_1,z_2)$, recall we conjugate the kernel by $(1-\rho_-)^z$, so that we are in the setting of Section (2). We still denote it by $K^l(z_1,z_2)$ to avoid unnecessary superscripts.
\begin{equation}
\begin{split}
K^l(z_1,z_2) &= \sum_{k=0}^{n_l-1}\Psi^n_k(z_1)\Phi^n_k(z_2)\\
&=(\sum_{k=0}^{n_{ll}-1}+\sum_{k=n_{ll}}^{n_l-1}) (R_tQ^{-k}(z_1,\delta(X_0(n-k))h^n_kR^{-1}_t(z_2).
\end{split}
\end{equation}
Since $X_0(n-k) =X_0^{ll}(n_{ll}-k),  h^{n,X_0}_k(0,z_2)= h^{n_{ll},X_0^{ll}}_k(0,z_2)$, the first summation is equivalent to  $(S_{-t,-n_{ll}})^*\overline{S}^{\mathrm{epi}(X^{ll}_0)}_{-t,n_{ll}}(z_1,z_2)$. Using the Proposition \eqref{mainconvergencetheoremleft}, the first part is
\begin{equation}\label{llconvergence}
\overline{1}^a (S_{-t,-n_{ll}})^*\overline{S}^{\mathrm{epi}(X^{ll}_0)}_{-t,n_{ll}}\overline{1}^a\to \overline{1}_{-\tilde{\bs}_-} (\bS_{-\bt,\tilde{\bx}_-})^*\bS^{\mathrm{epi}(-\mathfrak{h}_0^-)}_{-\bt,-\tilde{\bx}_-}\overline{1}_{-\tilde{\bs}_-}
\end{equation}
where $\mathfrak{h}_0^-$ is a Brownian motion starting from $0$.
The second part is 
\begin{equation}\label{lrformua}
\begin{split}
&\sum_{k=n_{ll}}^{n_l-1}R_tQ^{-k}(z_1,\delta(X_0(n-k))h^n_kR^{-1}_t(z_2)\\
&=\sum_{k=0}^{n_l-n_{ll}-1}R_tQ^{-n_{ll}-k}(z_1,\delta(X^{lr}_0(n_l-n_{ll}-k))h^{n_l,X_0^l}_{n_{ll}+k}R^{-1}_t(z_2).
\end{split}
\end{equation}
Using Lemma \eqref{splittinglemma}, 
\begin{equation}
h^{n_l,X_0^l}_{n_{ll}+k}(0,z_2)=\sum_{z_3}h^{n_{lr}}_k(0,z_3)(Q^{n_{ll}}(z_3,z_2)-G^{X_0^{ll}}_{0,n_{ll}}(z_3,z_2)),
\end{equation}
equation \eqref{lrformua} equal to
\begin{equation}
\begin{split}
\sum_{k=0}^{n_l-n_{ll}-1}\sum_{z_3}R_tQ^{-n_{ll}-k}(z_1,\delta(X_0(n-n_{ll}-k))h^{n_{lr}}_k(0,z_3)(Q^{n_{ll}}R^{-1}_t(z_3,z_2)-G^{X_0^{ll}}_{0,n_{ll}}R^{-1}_t(z_3,z_2)).
\end{split}
\end{equation}
We separate the sum in $z_3$ into two parts according to whether $z_3 \leq X^l_0(n_{lr})$ or $z_3 > X^l_0(n_{lr})$. In the case $z_3 \leq X^l_0(n_{lr})$, using Proposition \eqref{rightpiececonvergence},
\begin{equation}\label{eqn:82}
    \sum_{k=0}^{n_l-n_{ll}-1}\sum_{z_3\leq X^l_0(n_{lr})}R_tQ^{-n_{ll}-k}(z_1,\delta(X_0(n-n_{ll}-k))h^{n_{lr}}_k(0,z_3) \to (\bS_{-\bt,\tilde{\bx}_-}^{\mathrm{epi}(-\mathfrak{h}_0^{l,+})})^*.
\end{equation}
Using Proposition \eqref{prop:qqbarsameinlimit},
\begin{equation}\label{eqn:83}
    Q^{n_{ll}}R^{-1}_t \to \bS_{-\bt,-\tilde{\bx}_-}.
\end{equation}Notice that $G^{X_0^{ll}}_{0,n_{ll}}R^{-1}_t=\overline{S}^{\mathrm{epi}(X^{ll}_0)}_{-t,n_{ll}}$, we have by Proposition \eqref{mainconvergencetheoremleft},
\begin{equation}\label{eqn:84}
    G^{X_0^{ll}}_{0,n_{ll}}R^{-1}_t \to \bS_{-\bt,-\tilde{\bx}_-}^{\mathrm{epi}(-\mathfrak{h}_0^{l,-})}.
\end{equation}
$\mathfrak{h}_0^+$ is a Brownian motion independent of $\mathfrak{h}_0^-$. In the case $z_3 > X^l_0(n_{lr})$, we have $G^{X_0^{ll}}_{0,n_{ll}}(z_3,z_2)=\overline{Q}^{n_{ll}}(z_3,z_2)$. Using Proposition \eqref{prop:qqbarsameinlimit}, 
\begin{equation}
||Q^{n_{ll}}R_t^{-1}-\overline{Q}^{n_{ll}}R^{-1}_t||_{tr}\to 0.
\end{equation}
Thus, it is negligible in the limit.

Combining equation \eqref{llconvergence}, \eqref{eqn:82},\eqref{eqn:83},\eqref{eqn:84}, taking an adjoint and conjugating by a reflection operator, we derived that
\begin{multline}
\lim_{\varepsilon\to 0}\mathbb{E}_{X_0}\big[\det(I-\overline{1}^aM_{1-\rho_-}K^lM_{1/(1-\rho_-)}\overline{1}^a)_{l^2(\mathbb{Z})}\big] =
\mathbb{E}_{\mathfrak{h}^l}\big[\det\big(I-1_{\tilde{\bs}_-}\big((\bS^{\mathrm{hypo}(\mathfrak{h}_0^{l,-})}_{\bt,-\tilde{\bx}_-})^*\bS_{\bt,\tilde{\bx}_-}\\+(\bS_{\bt,-\tilde{\bx}_-})^*1^{\mathfrak{h}_0^{l}(0)}\bS^{\mathrm{hypo}(\mathfrak{h}_0^{l,+})}_{\bt,\tilde{\bx}_-}-(\bS^{\mathrm{hypo}(\mathfrak{h}_0^{l,-})}_{\bt,-\tilde{\bx}_-})^*1^{\mathfrak{h}^l_0(0)}\bS^{\mathrm{hypo}(\mathfrak{h}_0^{l,+})}_{\bt,\tilde{\bx}_-})1_{\tilde{\bs}_-}\big)_{L^2(\mathbb{R})}\big]
\end{multline}
Let's call the kernel being subtracted from $I$ in the Fredholm determinant $1_{\tilde{\bs}_-}\bK_{\tilde{\bx}_-}^{\mathfrak{h}_0^l}1_{\tilde{\bs_-}}$.
 
\subsection{\texorpdfstring{Convergence of $K^r$}{Convergence of Kr}}\label{subsec:convergencekr}
Now we look at kernel $K^r(z_1,z_2)$, we conjugate the kernel by $(1-\rho_+)^z$, thus although we use the same symbol as in the previous section, they are different.
\begin{equation}\label{rformula}
\begin{split}
K^r(z_1,z_2) &= \sum_{k=n_l}^{n-1}\Psi^n_k(z_1)\Phi^n_k(z_2)\\
&=(\sum_{k=n_l}^{n_{l}+n_{rl}-1}+\sum_{k=n_{l}+n_{rl}}^{n-1}) (R_tQ^{-k}(z_1,\delta(X_0(n-k))h^n_kR^{-1}_t(z_2)
\end{split}
\end{equation}
Considering the first part, we first extend the summation to $0$, which is 
\begin{equation}\label{eqn:extendto0}
(\sum_{k=0}^{n_{l}+n_{rl}-1}-\sum_{k=0}^{n_l-1})R_tQ^{-k}(z_1,\delta(X_0(n-k))h^n_k(0,z_2)R^{-1}_t.
\end{equation}
Recall $n_l+n_{rl}=n_+$, By Proposition \eqref{prop:mainconvergencetheoremright}, The first part is 
\begin{equation}\label{rlconvergence}
\overline{1}^a(S_{-t,-n_+})^*\overline{S}^{\mathrm{epi}(X^{l+rl}_0)}_{-t,n_+}\overline{1}^a \to \overline{1}_{-\tilde{\bs}_+}(\bS_{-\bt,\tilde{\bx}_+})^*\bS^{\mathrm{epi}(-\mathfrak{h}_0^{r,-})}_{-\bt,-\tilde{\bx}_+}\overline{1}_{-\tilde{\bs}_+}
\end{equation}
where $\mathfrak{h}_0^{r,-}$ is a Brownian motion. Notice that this looks like \eqref{llconvergence}, however, there is some difference. Here, the initial condition $X^{l+rl}_0$ consists of two parts: the $X^{rl}_0$ part is a Geo($\rho_+$) random walk; $X^{l}_0$ part is a Geo($\rho_-$) random walk. The point is that the hitting probability will only use the $X^{rl}_0$ part due to the scaling, which is explained in the proposition.

The second part is 
\begin{equation}
\overline{1}^a(S^{\rho_+}_{-t,-n_{l}})^*\overline{S}^{\rho_+,\mathrm{epi}(X^{l}_0)}_{-t,n_{l}}\overline{1}^a.
\end{equation}
This looks like $K^{l}$, but with a different conjugation. That is why we make the dependence on $\rho_+$ explicit. Using Proposition \eqref{prop:wrongconjuazero}, 
\begin{equation}
||\overline{1}^a(S^{\rho_+}_{-t,-n_{l}})^*\overline{S}^{\rho_+,\mathrm{epi}(X^{l}_0)}_{-t,n_{l}}\overline{1}^a||_{tr} \to 0.
\end{equation}
One short explanation is that it has the wrong scaling.

Then we look at the second sum in \eqref{rformula}, using the exact same treatment for the left part, \begin{equation}\label{rrformua}
\begin{split}
&\sum_{k=n_+}^{n-1}R_tQ^{-k}(z_1,\delta(X_0(n-k))h^n_kR^{-1}_t(z_2)\\
&=\sum_{k=0}^{n_{rr}-1}R_tQ^{-n_+-k}(z_1,\delta(X^{rr}_0(n_{rr}-k))h^{n}_{{n_+}+k}R^{-1}_t(z_2)
\end{split}
\end{equation}
Using the Lemma \eqref{splittinglemma}, 
\begin{equation}
h^{n,X_0}_{n_{+}+k}(0,z_2)=\sum_{z_3}h^{n_{rr}}_k(0,z_3)(Q^{n_+}(z_3,z_2)-G^{X_0^{l+rl}}_{0,n_{+}}(z_3,z_2)).
\end{equation}
Thus, equation \eqref{rrformua} equals to
\begin{equation}
\begin{split}
\sum_{k=0}^{n_{rr}-1}\sum_{z_3}R_tQ^{-n_+-k}(z_1,\delta(X^{rr}_0(n_{rr}-k))h^{n_{rr}}_k(0,z_3)(Q^{n_+}R^{-1}_t(z_3,z_2)-G^{X_0^{l+rl}}_{0,n_{+}}R^{-1}_t(z_3,z_2)).
\end{split}
\end{equation}
We separate the sum in $z_3$ into two parts according to whether $z_3 \leq X^{rr}_0(n_{rr})$ or $z_3 > X^{rr}_0(n_{rr})$. In the case $z_3 \leq X^{rr}_0(n_{rr})$, Using Proposition \eqref{rightpiececonvergence},
\begin{equation}\label{eqn:95}
 \sum_{k=0}^{n_{rr}-1}\sum_{z_3\leq X^l_0(n_{lr})}R_tQ^{-n_{+}-k}(z_1,\delta(X^{rr}_0(n_{rr}-k))h^{n_{rr}}_k(0,z_3) \to (\bS_{-\bt,\tilde{\bx}_+}^{\mathrm{epi}(-\mathfrak{h}_0^{r,+})})^*.
\end{equation}
Using Proposition \eqref{prop:qqbarsameinlimit},
\begin{equation}\label{eqn:96}
    Q^{n_{+}}R^{-1}_t \to \bS_{-\bt,-\tilde{\bx}_+}.
\end{equation}
Since $G^{X_0^{l+rl}}_{0,n_{+}}R^{-1}_t=\overline{S}^{\mathrm{epi}(X^{l+rl}_0)}_{-t,n_+}$, using
\eqref{prop:mainconvergencetheoremright}, we have
\begin{equation}\label{eqn:97}
    G^{X_0^{l+rl}}_{0,n_+}R^{-1}_t \to \bS_{-\bt,-\tilde{\bx}_+}^{\mathrm{epi}(-\mathfrak{h}_0^{r,-})}.
\end{equation}
$\mathfrak{h}_0^{r,+}$ is a Brownian motion independent of $\mathfrak{h}_0^{r,-}$. In the case $z_3 > X^{rr}_0(n_{rr})$, we have $G_{0,n_+}(z_3,z_2)=\overline{Q}^{n_{+}}(z_3,z_2)$. Using Proposition \eqref{prop:qqbarsameinlimit},
\begin{equation}
||Q^{n_{+}}R_t^{-1}-\overline{Q}^{n_{+}}R^{-1}_t||_{tr}\to0
\end{equation}
Thus it is negligible in the limit.

Combining equations \eqref{rlconvergence},\eqref{eqn:95},\eqref{eqn:96},\eqref{eqn:97}, taking an adjoint and conjugating by the reflection operator, we derive that:
\begin{multline}
\lim_{\varepsilon\to 0}\mathbb{E}_{X_0}\big[\det(I-\overline{1}^a M_{1-\rho_+}K^rM_{1/(1-\rho_+)}\overline{1}^a)_{l^2(\mathbb{Z})}\big] =
\mathbb{E}_{\mathfrak{h}_0^{r}}\big[\det\big(I-1_{\tilde{\bs}_+}\big((\bS^{\mathrm{hypo}(\mathfrak{h}_0^{r,-})}_{\bt,-\tilde{\bx}_-})^*\bS_{\bt,\tilde{\bx}_-}\\+(\bS_{\bt,-\tilde{\bx}_-})^*1^{\mathfrak{h}_0^r(0)}\bS^{\mathrm{hypo}(\mathfrak{h}_0^{r,+})}_{\bt,\tilde{\bx}_-}-(\bS^{\mathrm{hypo}(\mathfrak{h}_0^{r,-})}_{\bt,-\tilde{\bx}_-})^*1^{\mathfrak{h}^r_0(0)}\bS^{\mathrm{hypo}(\mathfrak{h}_0^{r,+})}_{\bt,\tilde{\bx}_-})1_{\tilde{\bs}_+}\big)_{L^2(\mathbb{R})}\big].
\end{multline}
Let us call the kernel being subtracted from $I$ in the Fredholm determinant $1_{\tilde{\bs}_+}\bK_{\tilde{\bx}_+}^{\mathfrak{h}^r}1_{\tilde{\bs}_+}$.
\subsection{Identify the limiting kernel}\label{sec:identifylimitingkernel}
Conclude from previous two sections, we have 
\begin{equation}\label{limitingformula}
\begin{split}
&\lim_{\varepsilon \to 0}\mathbb{P}(X_t(\rho_+\rho_-y_\varepsilon+\frac{\rho_-N_r+\rho_+N_l}{\rho_+-\rho_-}-\bs\varepsilon^{-1/2})\geq (1-\rho_+-\rho_-)y_\varepsilon-\frac{N_r+N_l}{\rho_+-\rho_-})\\
 &= \int \mu_W^l(d\mathfrak{h}^{l})\mu^r_{W}(d \mathfrak{h}^{r})\det(I-1_{\tilde{\bs}_-}\bK_{\tilde{\bx}_-}^{\mathfrak{h}^l}1_{\tilde{\bs}_-})_{L^2(\mathbb{R})}\det(I-1_{\tilde{\bs}_+}\bK_{\tilde{\bx}_+}^{\mathfrak{h}^r}1_{\tilde{\bs}_+})_{L^2(\mathbb{R})}
\end{split}
\end{equation}
where $\mu_W^l,\mu_W^r$ are independent Wiener measure on $C(\mathbb{R})$.
Now we want to identify these two integrations of Fredholm determinants. Using the following results from \cite{bfp2010}. 

\begin{theorem}\cite{bfp2010}
Let $\mu_\rho$ be the probability measure on $\{0,1\}^\mathbb{Z}$ such that $\mu_\rho(\eta_j = 1) = \rho$. Let $\chi=\rho(1-\rho)$. Let the rightmost particle on the negative integer sites be labeled $X(1)$. Define the following scaling
\begin{equation}
\begin{split}
n_\varepsilon &= \rho^2\varepsilon^{-3/2}\bt-2\rho \chi^{1/3}\varepsilon^{-2/2}\bx\\
r_\varepsilon &= (1-\rho)^2\varepsilon^{-3/2}\bt+2\chi^{1/3}\varepsilon^{-2/2}\bx-(1-\rho)\chi^{-1/3}\varepsilon^{-1/2}\bs.
\end{split}
\end{equation}
Then the one point distribution function of TASEP starting from stationary initial data has the following limit, 
\begin{equation}\label{stationarytasep}
 \lim_{\varepsilon \to 0}\mathbb{P}(X_{\varepsilon^{-3/2}\bt}(n_\varepsilon)\geq r_\varepsilon) = \frac{\partial}{\partial \bt^{-1/3}\bs}(g(\bt^{-2/3}\bx,\bt^{-1/3}\bs)\det(1-1_{\bt^{-1/3}\bs} K_{Ai}1_{\bt^{-1/3}\bs})_{L^2(\mathbb{R})})
 \end{equation} 
   $ K_{\rm Ai}$ is the so-called Airy kernel~\cite{J03}\cite{PS02} with shifted entries  defined by the kernel
\begin{equation}\label{eqKhat}
\begin{aligned}
& K_{\rm Ai}(x, y)=\int_0^\infty d\lambda \Ai(x+\lambda+\tau)\Ai(y+\lambda+\tau).
\end{aligned}
\end{equation}
The function $g(\tau,s)$ is defined by
\begin{equation}\label{eq:gm}
\begin{aligned}
g(\tau,s)&={\mathcal{R}}-\langle \rho P_s \Phi,P_s\Psi\rangle \\
&={\mathcal{R}}-\int_{s}^\infty dx\int_{s}^\infty dy \, \Psi(y) \rho(y,x) \Phi(x),
\end{aligned}
\end{equation}
where $\rho:=(\mathbbm{1}-P_s  K_{\rm Ai} P_s)^{-1}$. Finally the functions $\mathcal{R}$, $\Phi$, and $\Psi$ are defined as
\begin{equation}\label{eq1.12}
\begin{aligned}
{\mathcal{R}} =&\, s+e^{-\frac23\tau^3} \int_{s}^\infty dx\int_0^\infty dy\, \Ai(x+y+\tau^2) e^{-\tau(x+y)},\\
\Psi(y)=&\, e^{\frac23\tau^3+\tau y}-\int_0^\infty dx\, \Ai(x+y+\tau^2)e^{-\tau x},\\
\Phi(x)=&\, e^{-\frac23\tau^3}\int_0^\infty d\lambda\int_{s}^\infty dy\, e^{-\lambda(\tau-\tau)}e^{-\tau y} \Ai(x+\tau^2+\lambda) \Ai(y+\tau^2+\lambda)\\
+& -\int_{0}^\infty dy\, \Ai(y+x+\tau^2)e^{\tau y},
\end{aligned}
\end{equation}
where $\Ai$ denotes the Airy function.
\end{theorem}
The right-hand side is the one-point distribution for the $\Aistat$ process. At $\bx=0,\bt=1$, it is the Baik-Rains distribution \cite{BR00}\cite{FS06}. On the left-hand side, it is almost the same scaling that we took, but with a different constant coefficient. In our language, it is $\int \mu_W(d\mathfrak{h})\det(I-1_{\bt^{-1/3}\bs} \bK^{\mathfrak{h}}_{\bt^{-2/3}\bx}1_{\bt^{-1/3}\bs})$. The scaling limit is almost exactly the same as in Proposition \eqref{mainconvergencetheoremleft},\eqref{prop:mainconvergencetheoremright}, except that there is no constant coefficient on $\bs,\bx$. Using this theorem, we can conclude that \eqref{limitingformula} is the product of $F_{\bt^{-2/3}\tilde{\bx}_-}(\bt^{-1/3}\tilde{\bs}_-)$ and $F_{\bt^{-2/3}\tilde{\bx}_+}(\bt^{-1/3}\tilde{\bs}_+)$, which completes our proofs.

\section{Convergence Results}
We first state our main convergence theorem for the ``good" part of the operators. This is the type of standard steepest descent analysis for models in the KPZ universality class. The proof is similar to the proof in \cite{MQR2021}, adapted to the density $\rho$ case. The proof of uniform boundedness on the trace norm is technical, which will be left to the appendix.
\begin{proposition}\label{mainconvergencetheoremleft}
Assume that we scale the variables in the following way. 
    \begin{equation}
    \begin{split}
    &n_- = \rho_-^2\varepsilon^{-3/2}\bt+\rho_-\rho_+\varepsilon^{-1}-\varepsilon^{-1/2}\bs, t = \varepsilon^{-3/2}\bt+\varepsilon^{-1}\bx\\
    &a = (1-\rho_--\rho_+)(\varepsilon^{-3/2}\bt+\varepsilon^{-1}\bx)-\frac{N_r+N_l}{\rho_+-\rho_-}.
    \end{split}
    \end{equation}
    We also scale $x_1 = \varepsilon^{-3/2}(\rho_--\rho_+)\bt-\frac{N_r+N_l}{\rho_+-\rho_-}+\varepsilon^{-1/2}\chi_-^{-1/3}(1-\rho_-)\bv, x_2 = a+\varepsilon^{-1/2}\chi_-^{-1/3}(1-\rho_-)\bu$. As $\varepsilon \to 0$,
    \begin{equation}
    \begin{split}
        &\varepsilon^{-1/2}\chi_-^{-1/3}(1-\rho_-)S_{-t,-n_-}(x_1,x_2) \rightarrow \mathbf{S}_{-\bt,\bx}(\bv,\bu),\\
    &\varepsilon^{-1/2}\chi_-^{-1/3}(1-\rho_-)\overline{S}_{-t,n_-}(x_1,x_2) \rightarrow \mathbf{S}_{-\bt,-\bx}(\bv,\bu).\\
    \end{split}
    \end{equation}
\begin{proposition}\label{epiconvergenceleft}
Under the scaling of previous proposition, 
\begin{equation}
\varepsilon^{-1/2}\chi_-^{-1/3}(1-\rho_-)\overline{S}_{-t,n_-}^{\mathrm{epi}(X^{ll}_0)}(x_1,x_2) \to \bS_{-\bt,-\bx}^{\mathrm{epi}(-\mathfrak{h}_0^{l,-})}(\bv,\bu)\\
\end{equation}
where $\mathfrak{h}^{l,-}_0$ is a Brownian motion.
\end{proposition}
\begin{proof}
Recall that \begin{equation}
\overline{S}_{-t,n}^{\mathrm{epi}(X_0)}(x_1,x_2)  = \mathbb{E}_{B_0=x_1}[\overline{S}_{-t,n-\tau}(B_\tau,x_2)1_{\tau<n}]
\end{equation}
Now we want to investigate what does $X_0^{ll,\varepsilon}$ converges to. For each fixed $X_0$, $X_0^{ll,\varepsilon}$ does not converge, since the starting place of $X_0^{ll,\varepsilon}$ changes for every $\varepsilon$. However, it converges in distribution to the Brownian motion. By our scaling, the expectation is over a geometric random walk starts at $\varepsilon^{-3/2}(\rho_--\rho_+)\bt-\frac{N_r+N_l}{\rho_+-\rho_-}+\varepsilon^{-1/2}\bv$, jumps to the left with the step size being Geo$(\rho_-)$. By our splitting of the formula,  $X^{ll,\varepsilon}_0(1)$ is $X_0( n_r+\rho_-(\rho_--\rho_+)\varepsilon^{-3/2}\bt+\frac{\rho_-N_r+\rho_+N_l}{\rho_+-\rho_-})$. Write it as $$X_0( n_r+\rho_-(\rho_--\rho_+)\varepsilon^{-3/2}\bt+N_l+(\frac{\rho_-N_r+\rho_+N_l}{\rho_+-\rho_-}-N_l)).$$ Notice that $X_0(n_r+\rho_-(\rho_--\rho_+)\varepsilon^{-3/2}\bt+N_l)$ is the particle exactly at position $\varepsilon^{-3/2}(\rho_--\rho_+)\bt$ (with $O(1)$ error). And counting the mean contribution from the term $(\frac{\rho_-N_r+\rho_+N_l}{\rho_+-\rho_-}-N_l)$, which is $-\frac{N_r+N_l}{\rho_+-\rho_-}$. The place of $X^{ll}_0(1)$ on order greater than or equal $\varepsilon^{-3/4}$ is
\begin{equation}
\varepsilon^{-3/2}(\rho_--\rho_+)\bt-\frac{N_r+N_l}{\rho_+-\rho_-}.
\end{equation}
So, the place of the initial particle matches the starting point of the random walk on the $\varepsilon^{-3/2}$ terms and the $\varepsilon^{-3/4}$ term. Thus, the formula only depends on the relative position of $X^{ll}_0(1)$ and the random walk on the $\varepsilon^{-1/2}$ scale. We are looking at the scaling of the initial configuration starting from $X_0(n_r+\rho_-(\rho_--\rho_+)\varepsilon^{-3/2}\bt+\frac{\rho_-N_r+\rho_+N_l}{\rho_+-\rho_-})$ to the left and right. On the left, the increments are i.i.d Geo($\rho_-$) random variables since the particles after $n_r+\rho_-(\rho_--\rho_+)\varepsilon^{-3/2}+\frac{\rho_-N_r+\rho_+N_l}{\rho_+-\rho_-}$ are independent of the random variables $N_r,N_l$. Now we consider the same scaling for the walk and initial condition: $B_\varepsilon(\bx)=\varepsilon^{1/2}(B_{\varepsilon^{-1}\bx}+\varepsilon^{-1}\bx/\rho_--1)$ and $X_0^{ll,\varepsilon}=\varepsilon^{1/2}(X^\varepsilon_0(\varepsilon^{-1}\bx)+\varepsilon^{-1}\bx/\rho_--1)$. Let $\tau^\varepsilon$ be the hitting time by $B_\varepsilon$ of $\mathrm{epi}(-X_0^{ll,\varepsilon})$ By the Donsker theorem, $B_\varepsilon$ converges locally uniformly in distribution to a Brownian motion with diffusion coefficient 2, and $\tau^\varepsilon $ converges to the hitting time of a Brownian motion, using the following proposition proved in \cite{MQR2021}.

\end{proof}
\end{proposition}
\begin{proposition}\cite{MQR2021}\label{hittingprobabilityconvergence}
Suppose $g_\varepsilon \to g$ locally in LC. Let $B(x),x>0$ be a Brownian motion starting at $z < g(0)$ and let $B_\varepsilon(x)$ be stochastic processes with $B_\varepsilon \to B$ in UC, in distribution. Let $\tau^\varepsilon = \inf\{x \geq 0: B(x)\geq g(x)\}$ be the first hitting times of the $\mathrm{epi}(g_\varepsilon)$ and $\mathrm{epi}(g)$. Then $\tau^\varepsilon \to \tau$ in the distribution. Furthermore, convergence is uniform over $g$ in the set of bounded H\"older $\beta$-norm, $\beta \in (0,1/2)$.
\end{proposition}
Now we state the result for the convergence for the right half, the scaling is almost the same as in Proposition \eqref{mainconvergencetheoremleft}.
\begin{proposition}\label{prop:mainconvergencetheoremright}
    We consider the following scaling.
    \begin{equation}
    \begin{split}
    &n_+ = \rho_+^2\varepsilon^{-3/2}\bt+\rho_-\rho_+\varepsilon^{-1}-\varepsilon^{-1/2}\bs, t = \varepsilon^{-3/2}\bt+\varepsilon^{-1}\bx\\
    &a = (1-\rho_--\rho_+)(\varepsilon^{-3/2}\bt+\varepsilon^{-1}\bx)-\frac{N_r+N_l}{\rho_+-\rho_-}.
    \end{split}
    \end{equation}
    If we set $x_1 = \varepsilon^{-3/2}(\rho_+-\rho_-)\bt+\frac{N_r-N_l}{\rho_+-\rho_-}+\varepsilon^{-1/2}\chi_+^{-1/3}(1-\rho_+)\bv, x_2 = a+\varepsilon^{-1/2}\chi_+^{-1/3}(1-\rho_+)\bu$. As $\varepsilon \to 0$,
    \begin{equation}\label{eqn:103}
    \begin{split}
        \varepsilon^{-1/2}\chi_+^{-1/3}(1-\rho_+)S_{-t,-n_+}(x_1,x_2) \rightarrow \mathbf{S}_{-\bt,\bx}(\bv,\bu),\\
    \varepsilon^{-1/2}\chi_+^{-1/3}(1-\rho_+)\overline{S}_{-t,n_+}(x_1,x_2) \rightarrow \mathbf{S}_{-\bt,-\bx}(\bv,\bu).
    \end{split}
    \end{equation}
\end{proposition}
We also have a similar proposition for $S^{\mathrm{epi},\varepsilon}$.
\begin{equation} \label{epiconvergenceright}
\begin{split}
&\varepsilon^{-1/2}\chi_+^{-1/3}(1-\rho_+)\overline{S}_{-t,n_+}^{\mathrm{epi}(X^{l+rl}_0)}(x_1,x_2) \to \bS_{-\bt,-\bx}^{\mathrm{epi}(-\mathfrak{h}_0^{r,-})}(\bv,\bu)
\end{split}
\end{equation}
where $\mathfrak{h}_0^{r,-}$ are independent Brownian motion starting from $\bv$.

\begin{proof}
The proof of the convergence of \eqref{eqn:103} is exactly the same as Proposition \eqref{mainconvergencetheoremleft}. Here we discuss the convergence in \eqref{epiconvergenceright}.
The difference from Proposition \eqref{epiconvergenceleft} is that the initial condition contains two parts: $X_0^l$, which is a Geo($\rho_-$) random walk; $X_0^{rl}$, which is a Geo($\rho_+$) random walk. Also, for any fixed $\varepsilon$, the particles on the left of $X_0(n^{rr})$ are not independent of $N_r$. However, both problems do not matter since the hitting probability cares only about particles around $X_0(n^{rr})$ on the $\varepsilon^{-1}$ scale. More precisely, recall 
\begin{equation}
\overline{S}_{-t,n_+}^{\mathrm{epi}(X^{l+rl}_0)}(x_1,x_2) = \sum_{k=0}^{n_+}\sum_{B(\tau)>X_0(s+1)} \mathbb{P}_{B(0)=x_1}(\tau =k, B(\tau) = b\varepsilon^{-1/2})\overline{S}_{-1,n_+-s}(B(\tau),x_2)
\end{equation}
Fix $d >0$, for any $k>\varepsilon^{-1}d$, 
\begin{equation}\label{tailprobabilitytracenorm}
\begin{split}
 &||\sum_{k=\varepsilon^{-1}d}^{n_+}\sum_{b\varepsilon^{-1/2}>X_0(k+1)} \mathbb{P}_{B(0)=x_1}(\tau =k, B(\tau) = b\varepsilon^{-1/2})\overline{S}_{-1,n_+-k}(B(\tau),x_2)||_{tr}\\
 &< \mathbb{P}_{B(0)=x_1}(\text{no hit in } \varepsilon^{-1}d \text{ steps } )\sup_{k,B(\tau)}||\overline{S}_{-1,n_+-k}(B(\tau),x_2)||_{L^2}
\end{split}
\end{equation}
The supremum is finite since $||\overline{S}_{-1,n_+-k}(B(\tau),z_2)||_{L^2}$ is uniformly bonded. Thus, in the limit as $\varepsilon \to 0, d\to \infty$, the trace norm of \eqref{tailprobabilitytracenorm} goes to $0$. 

Now, look at the sum from $0$ to $\varepsilon^{-1}d$. By our scaling, the expectation is over a geometric random walk starts at $\varepsilon^{-3/2}(\rho_+-\rho_-)\bt-\frac{N_r+N_l}{\rho_+-\rho_-}+\varepsilon^{-1/2}\bv$, jumps to the left with the step size being Geo$(\rho_+)$. By our splitting of the formula,  $X^{l+rl}_0(1)$ is $X_0( n_{rr})$. The place of $X_0(n_{rr})$ on order greater than or equal to $\varepsilon^{-3/4}$ is $$(\rho_+-\rho_-)\varepsilon^{-3/2}\bt-\frac{N_r+N_l}{\rho_+-\rho_-},$$ that matches the starting place of the random walk. Then we want to determine the limiting distribution of $Y = (X_0(n_{rr}+1),X_0(n_{rr}+2),\cdots,X_0(n_{rr}+d\varepsilon^{-1}))$ given the value of $N_r$. Notice $N_r$ depends only on the last element $X_0(n_{rr}+d\varepsilon^{-1})$, we can see that they are asymptotically independent, since for any $a,b\in\mathbb{R},\delta_1,\delta_2>0$,\begin{equation}
\lim_{\varepsilon\to 0}\mathbb{P}(\varepsilon^{3/4}N_r \in (a,a+\delta_1)|\varepsilon^{1/2}X(n_{rr}+d\varepsilon^{-1}) \in (b,b+\delta_2)) =\lim_{\varepsilon \to 0}\mathbb{P}(\varepsilon^{3/4}N_r \in (a,a+\delta_1)),
 \end{equation} 
 which is true since a.s., $\lim_{\varepsilon \to 0}\varepsilon^{3/4}X_0(n_{rr}+d\varepsilon^{-1}) = 0$. So using the convergence result from equation \eqref{eqn:103} and the Donsker theorem, 
\begin{equation}
\mathbb{E}_{B_0=x_1}[\overline{S}_{-t,n-\tau}(B_\tau,x_2)1_{\tau<\varepsilon^{-1}s}]\to \mathbb{E}_{\bB_0=\bv}[\bS_{\bt,\bx-\tau}(\bB(\btau,\bu)1_{\btau<d})].
\end{equation}
The expectation is over a Brownian motion hitting another Brownian motion which is $v$ distance apart at $0$. Then sending $d\to \infty$, we get the desired result.
\end{proof}

\begin{proposition}\label{rightpiececonvergence}
Let $z_1= a+\varepsilon^{-1/2}\chi_-^{-1/3}(1-\rho_+)\bu,z_3 =\varepsilon^{-3/2}(\rho_--\rho_+)\bt-\frac{N_r+N_l}{\rho_+-\rho_-}+\varepsilon^{-1/2}\chi_-^{-1/3}(1-\rho_-)\bv $
\begin{equation}
\sum_{k=0}^{n_l-n_{ll}-1}\sum_{z_3\leq X^l_0(n_{lr})}R_tQ^{-n_{ll}-k}(z_1,\delta(X_0(n-n_{ll}-k))h^{n_{lr}}_k(0,z_3) \to (\bS_{-\bt,\tilde{\bx}_-}^{\mathrm{epi}(-\mathfrak{h}_0^{l,+})})^*(\bu,\bv)
\end{equation}
\end{proposition}

\begin{proof}
Left-hand side equals to 
\begin{equation}
 \lim_{c \to \infty}R_tQ^{-n_{ll}-c\varepsilon^{-1}}(z_1,z_2)(\sum_{k=0}^{c\varepsilon^{-1}-1}+\sum_{k=c\varepsilon^{-1}}^{n_l-n_{ll}-1})\sum_{z_3\leq X^l_0(n_{lr}),z_2}Q^{c\varepsilon^{-1}-k}(z_2,\delta(X_0(n-n_{ll}-k))h^{n_{lr}}_k(0,z_3)
 \end{equation} 
We will prove that in the limit of $\varepsilon \to 0$, the first sum converges to 
\begin{equation}\label{eqn:116}
\lim_{c\rightarrow \infty}(\bS_{-\bt,\tilde{\bx}_--c})^*(\bS_{[0,c]}^{\mathrm{epi}(-\mathfrak{h}^{l,+})})^*
\end{equation}
which is equal to $(\bS_{-\bt,\tilde{\bx}_-}^{\mathrm{epi}(-\mathfrak{h}_0^{l,+})})^*$ by equation \eqref{eqn:sepiseconddef}. We scale the variable $z_2$ to \[z_2 = \varepsilon^{-3/2}(\rho_--\rho_+)\bt-\frac{N_r+N_l}{\rho_+-\rho_-}+\frac{c\varepsilon^{-1}}{\rho_-}+\varepsilon^{-1/2}\chi_-^{-1/3}(1-\rho_-)\bv'.\] The extra $\frac{c\varepsilon^{-1}}{\rho_-}$ is used to cancel the shift from $c\varepsilon^{-1}$ on the exponent of $Q$. From Proposition \eqref{mainconvergencetheoremleft},
\begin{equation}\label{eqn:117}
    R_tQ^{-n_{ll}-c\varepsilon^{-1}}(z_1,z_2)=S_{t,-n_{ll}-c\varepsilon^{-1}} \to (\bS_{-\bt,\tilde{\bx}_--c})^*(\bu,\bv')
\end{equation}
    For each fixed $c$, when $z_3\leq X^l_0(n_{lr})$, 
    \begin{equation}
    \sum_{k=0}^{c\varepsilon^{-2/2}-1}Q^{c\varepsilon^{-2/2}-k}\delta_{x_0(N_2-n_s-k)}(z_2)h^{N_2-n_s}_{k}(0,z_3)
    \end{equation}
    is the probability that the Geo($\rho_-$) random walk starts at $z_3$ at time $-1$, end up at $z_2$ after $c\varepsilon^{-2/2}$ steps, having hit the curve $X^l_0(n_{lr})$ in between. Notice that the term $\frac{c\varepsilon^{-1}}{\rho_-}$ is the mean of the walk. Both the initial condition and the random walk scale diffusively, and they have the same mean, thus, by the Donsker theorem, it converges to $(\bS^{\mathrm{epi}(-\mathfrak{h}^{l,+})}_{[0,c]})^*(\bv',\bv)$. Combining with equation \eqref{eqn:117}, we proved \eqref{eqn:116}. Now we prove that the rest of the sum goes to zero. The rest of the sum is
    \begin{equation}
    (S_{-t,-n_l})^*\sum_{k=c\varepsilon^{-2/2}}^{n_l-n_{ll}-1}Q^{n_l-n_{ll}-1-k}\delta_{x_0(N_2-n_s-k)}(z_1)h^{n_{lr}}_{k}(0,z_3).
    \end{equation}
    The sum is the probability that the walk starts at $z_3$ at time $-1$, end up at $z_1$ after $N_2-n_s$ steps, having first hit the curve among $\varepsilon^{-1}c$ to $n_l-n_{ll}-1$ steps. This probability is less than the probability that the random walk does not hit in first $\varepsilon^{-1}c$ steps, which in the limit converge to the probability that the Brownian does not hit another Brownian motion which are finite distance apart at $t=0$, on interval $[0,s]$. This probability goes to $0$ as $s\to \infty$.
    \end{proof}

\begin{proposition}\label{prop:qqbarsameinlimit}Under the scaling in Proposition \eqref{mainconvergencetheoremleft},
\begin{equation}
\begin{split}
||Q^{n_{ll}}R^{-1}_t-\overline{Q}^{n_{ll}}R^{-1}_t||_{tr}\to 0
\end{split}
\end{equation}
Using this fact, we derive the scaling formula for $Q^{n_{ll}}R_t^{-1}$, 
\begin{equation}
    Q^{n_{ll}}R^{-1}_t \to \bS_{-\bt,-\tilde{\bx}_-}.
\end{equation}
\end{proposition}
\begin{proof}
Recall from equation (\ref{formulaSandSbar})
\begin{equation}
(\overline{Q}^{(n)}e^{(1-\rho)t\nabla^-})(z_1,z_2)= \frac{1}{2\pi i}\int_{\Gamma_0}dw (\frac{\rho}{1-\rho})^{n-1}\frac{(1-w)^{z_2-z_1+n-1}}{(1-\rho)^{z_2-z_1}w^n}e^{t(w-\rho)}\\
\end{equation}
Now we compute the $Q^n e^{\frac{t}{2}\nabla^-}$, it is
\begin{equation}
\begin{split}
Q^{n}e^{(1-\rho)t}(z_1,z_3)-\int_{\Gamma_0} dv\int_{\Gamma_0^C} dw (\frac{\rho}{1-\rho})^{n-1}\frac{-e^{-t(w-1/2)}}{(1-\rho)^{z_3-z_1}v^n(w+wv-1)w^{z_1-z_3}}
\end{split}
\end{equation}
where both the contour $\Gamma_0,\Gamma_0^C$ is a circle around $0$ such that $|w|=C,C>1, |v|<\varepsilon$, so contour for $v$ only has residues at $v=0$, which is :
\begin{equation}
\int dw (-1)^n(\frac{\rho}{1-\rho})^{n-1}\frac{e^{-t(w-\rho)}w^{n+z_3-z_1}}{(1-\rho)^{z_3-z_1}(w-1)}
\end{equation}
$n+z_3-z_1$ is positive for small enough $\varepsilon$, thus we only need to consider the residue at $w=1$, which is 
\begin{equation}
(\frac{\rho}{1-\rho})^{n-1}\frac{e^{-t(1-\rho)}}{(1-\rho)^{z_3-z_1}}.
\end{equation}
Recall the scaling of the variables; the leading term for $t$ is $\varepsilon^{-3/2}$, the leading term of $\frac{1}{(1-\rho)}$ is $(1-\rho)\varepsilon^{-3/2}$, thus the operator goes to $0$ in the trace norm as $\varepsilon \to 0$. Now using the convergence result about $\overline{Q}^{n_{ll}}R_t^{-1}$, we have the proposition.
\end{proof}
\begin{proposition}\label{prop:wrongconjuazero} Under the scaling of Proposition \eqref{mainconvergencetheoremleft},
\begin{equation}
||\overline{1}^a(S^{\rho_+}_{-t,-n_{l}})^*\overline{S}^{\rho_+,\mathrm{epi}(X^{l}_0)}_{-t,n_{l}}\overline{1}^a||_{tr} \to 0
\end{equation}
\end{proposition}
\begin{proof}
From Theorem (\ref{thm:convergenceandbound}), we know 
\begin{equation}
||\overline{1}^a(S^{\rho_-}_{-t,-n_{l}})^*\overline{S}^{\rho_-,\mathrm{epi}(X^{l}_0)}_{-t,n_{l}}\overline{1}^a||_{tr}
\end{equation}
is uniformly bounded, so we write
\begin{equation}
\begin{split}
||\overline{1}^a(S^{\rho_+}_{-t,-n_{l}})^*\overline{S}^{\rho_+,\mathrm{epi}(X^{l}_0)}_{-t,n_{l}}\overline{1}^a||_{tr} &= ||\overline{1}^aM_{\frac{(1-\rho_+)}{1-\rho_-}}(S^{\rho_-}_{-t,-n_{l}})^*\overline{S}^{\rho_-,\mathrm{epi}(X^{l}_0)}_{-t,n_{l}}M_{\frac{(1-\rho_-)}{1-\rho_+}}\overline{1}^a||_{tr}
\end{split}
\end{equation}
Define $c_1 = \frac{1-\rho_+}{1-\rho_-}<1,c_2=\frac{1-\rho_-}{1-\rho_+}>1$.
Recall the definition of $\overline{S}^{\rho_-,\mathrm{epi}(X_0)}_{-t,n_l}$,
\begin{equation}\label{tracenormformula}
 \begin{split}
 &\overline{1}^aM_{c_1}(S^{\rho_-}_{-t,-n_l})^*\overline{S}^{\rho_-,\mathrm{epi}(X_0^l)}_{-t,n_l}M_{c_2}\overline{1}^a(z_1,z_2)\\
 &=\int_{z\in \varepsilon\mathbb{Z},s\in[0,\varepsilon n)}\mathbb{P}_{B_\varepsilon(0)=z}(\tau_\varepsilon \in ds)c_1^{z_1}S^{\rho_-}_{-t,-n_l}(z,z_1)\overline{S}^{\rho_-}_{-t,n_l}(h_\varepsilon(s),z_2)c_2^{z_2}
 \end{split}
 \end{equation}
 where $h_\varepsilon(s)$ is the scaled random walk. Thus, 
 \begin{equation}
 \begin{split}
 \lVert  &\overline{1}^aM_{c_1}(S^\varepsilon_{-1,x_1})^*\overline{S}^\varepsilon_{-1,x_2}M_{c_2}\overline{1}^a\rVert_{tr}\\
 &\leq \int_{z\in \varepsilon\mathbb{Z},s\in[0,\varepsilon n)}\mathbb{P}_{B_\varepsilon(0)=z}(\tau_\varepsilon \in ds)\lVert c_1^{z_1}\overline{1}^aS^{\rho_-}_{-t,-n_l}(z,z_1)\overline{S}^{\rho_-}_{-t,n_l}\overline{1}^a(h_\varepsilon(s),z_2)c_2^{z_2}\rVert_{tr}
 \end{split}
 \end{equation}
 Since $\overline{1}^aS^{\rho_-}_{-t,-n_l}(z,z_1)\overline{S}^{\rho_-}_{-t,n_l}\overline{1}^a(h_\varepsilon(s),z_2)$ is a rank one operator, its trace norm is the product of $L^2$ norm,
 \begin{equation}
  \begin{split}
   &\lVert c_1^{z_1}\overline{1}^aS^{\rho_-}_{-t,-n_l}(z,z_1)\overline{S}^{\rho_-}_{-t,n_l}\overline{1}^a(h_\varepsilon(s),z_2)c_2^{z_2} \rVert_{1}\\
   &=(\int_{-\infty}^a(c_1^{z_1}S^{\rho_-}_{-t,-n_l}(z-z_1))^2dz_1\int_{\infty}^a (c_2^{z_2}\overline{S}^{\rho_-}_{-t,n_l}(h_\varepsilon(s)-z_2))^2dz_2)^{1/2}
  \end{split}
  \end{equation} 
\begin{multline}\label{eqn:133}
    \int_{-\infty}^adz_1 |S^{\rho_-}_{-t,-n_l}(z-z_1)|^2 =\\\frac{1}{(2\pi i)^2}\int_{C^0_\varepsilon}dw_1\int_{C^0_\varepsilon}dw_2 \varepsilon^{1/2}\frac{e^{F_1(t,x,z,\varepsilon,w_1)+
F_1(t,x,z,\varepsilon,w_2)}}{(1-c_1^2)+c_1^2\varepsilon^{1/2}w_1+c_1^2\varepsilon^{1/2}w_2-c_1^2\varepsilon^{1}w_1w_2}
\end{multline}

 where $C^0_\varepsilon$ is a circle centered at $\varepsilon^{-1/2}$ with radius $\varepsilon^{-1/2}$, with a little perturb at 0 so that 0 is outside the contour. Thus, the denominator will not be zero if $\varepsilon$ is small enough. From the proof in the Appendix, we know that \[||\frac{1}{(2\pi i)^2}\int_{C^0_\varepsilon}dw_1\int_{C^0_\varepsilon}dw_2 \frac{e^{F_1(t,x,z,\varepsilon,w_1)+
F_1(t,x,z,\varepsilon,w_2)}}{(1-c_1^2)+c_1^2\varepsilon^{1/2}w_1+c_1^2\varepsilon^{1/2}w_2-c_1^2\varepsilon^{1}w_1w_2}||_{tr}<\infty.\] Thus the extra $\varepsilon^{1/2}$ makes the trace norm vanish. For the other term, 
 \begin{equation}\label{eqn:134}
\begin{split}
     &\int_{-\infty}^adz_2 |\overline{S}^{\rho_-}_{-t,n_l}(h_\varepsilon(s)-z_2)|^2 \\
 &= \frac{1}{(2\pi i)^2}\int_{C_\varepsilon^0}dw_1\int_{C_\varepsilon^0}dw_2\varepsilon^{1/2}\frac{e^{F_1(t,x,h_\varepsilon(s),\varepsilon,w_1)+
F_1(t,x,h_\varepsilon(s),\varepsilon,w_2)}c_2^2(1+\varepsilon^{1/2}w_1)(1+\varepsilon^{1/2}w_2)}{c_2^2\varepsilon^{1/2}w_1+c_2^2\varepsilon^{1/2}w_2+c_2^2\varepsilon^{1}w_1w_2+(c^2-1)}
\end{split}
\end{equation}
For the same reason above, the trace norm of this part also goes to $0$.
\end{proof}

\section{The second class particle}
Let us first define the place of second class particle precisely. Let $\eta_0$ be the initial condition with $\eta_0(x)\sim \mathrm{Ber}(\rho_-)$ for $x<0$, $\eta_0(x) \sim \mathrm{Ber}(\rho_+)$ for $x>0$, and $\eta(x) = 0$. Let $\tilde{\eta}_0(x)=\eta_0(x)$ for $x\neq 0$ and $\tilde{\eta}_0(0)=1$. The particle at $0$ is the second class particle. We couple two initial conditions $\eta_0$ and $\tilde{\eta}_0$ using graphical construction. For more details on coupling, see \cite{liggett1999sip} \cite{Pablo18}. Notice that the only position where the two configurations are different at time $t$ is the position of the second-class particle. Thus, we have
\[X^{\mathrm{2nd}}(t) = \sum_{x\in \mathbb{Z}}x 1_{\eta_t(x)\neq \tilde{\eta}_t(x)}.\] Let $h(t,x)$ and $\tilde{h}(t,x)$ be the height functions associated with $\eta, \tilde{\eta}$, with the initial condition $h(0,0)=0$ and $\tilde{h}(0,0)=0$. Then we define the following modified version of height function:
\begin{equation}
\begin{aligned}
h^-(0,x)&=h(0,x)\mathbbm{1}_{x<0}+|x|\mathbbm{1}_{x\geq 0},\\
h^+(0,x)&=h(0,x)\mathbbm{1}_{x\geq 0}+|x|\mathbbm{1}_{x<0},\\
\tilde h^+(0,x)&=\tilde h(0,x)\mathbbm{1}_{x\geq 0}+|x|\mathbbm{1}_{x<0}=(h(0,x)-2)\mathbbm{1}_{x\geq 1}+|x|\mathbbm{1}_{x\leq 0}.
\end{aligned}
\end{equation}
$h^-(0,x)$ is the height function corresponding to the initial condition that all particles in the positive integer sites in $\eta_0$ are removed; $h^+(0,x),\tilde{h}^+(0,x)$, are the height functions corresponding to the initial condition that all particles in the negative integer sites in $\eta_0,\tilde{\eta}_0$ are occupied.
Now we can cite the following theorem from \cite{FN24}.
\begin{theorem}
Let us consider a generic initial condition $h(x,0)$ (it could be even random) satisfying the following three assumptions:
\begin{itemize}
\item[(a)] \emph{Initial macroscopic shock around the origin}: assume that for some $0<\rho_-<\rho_+<1$,
\begin{equation}
\lim_{x\to-\infty} x^{-1}h(x,0)=1-2\rho_-,\quad \lim_{x\to\infty} x^{-1}h(x,0)=1-2\rho_+.
\end{equation}
\item[(b)] \emph{Control of the end-point of the backwards geodesics} starting at $v_s t$ with $v_s=1-\rho_--\rho_+$: let $\bx^+$ (resp.\ $\bx^-$) any backwards geodesics starting at $(v_s t,t)$ with initial profile $\tilde h^+(\cdot,0)$ (resp.\ $h^-(\cdot,0)$). Then assume for some $\delta>0$, we have
\begin{equation}
\lim_{t\to\infty}\mathbb{P}(\bx^-(0)\leq -\delta t,\bx^+(0)\geq \delta t) = 1.
\end{equation}
\item[(c)] \emph{Limit process under KPZ scaling} of $\tilde h^+,h^-$ with \emph{constant law along characteristics}.  There exist limiting processes ${\mathcal H}^-(u)$ and ${\mathcal H}^+(u)$ whose distribution is continuous in $u$, such that
\begin{equation}\label{eq2.3}
\begin{aligned}
&\lim_{t\to\infty} \frac{h^-(t,v_s t+u t^{2/3})-{\bf h}^-(t,u)}{-t^{1/3}}={\mathcal H}^-(u),\\
&\lim_{t\to\infty} \frac{\tilde h^+(t,v_s t+u t^{2/3})-{\bf h}^+(t,u)}{-t^{1/3}}={\mathcal H}^+(u),
\end{aligned}
\end{equation}
where
\begin{equation}\label{eq2.4}
\begin{aligned}
{\bf h}^-(u,t)&=(1-\rho_--\rho_++2\rho_- \rho_+) t+ (1-2\rho_-) u t^{2/3},\\
{\bf h}^+(u,t)&=(1-\rho_--\rho_++2\rho_- \rho_+) t+ (1-2\rho_+) u t^{2/3}
\end{aligned}
\end{equation}
are what we expect macroscopically from the solution of the Burgers equation. Below we will use the notations
\begin{equation}
v_s=1-\lambda-\rho+2\lambda \rho,\quad \chi_-=\lambda(1-\lambda),\quad \chi_+=\rho(1-\rho).
\end{equation}
\end{itemize}

\label{thmMain}Let $\tau,s \in \R$.
Under the above assumptions, with $X^{\rm 2nd}(t)$ denoting the position of the second class particle at time $t$ starting from the origin, we have
\begin{equation}\label{eq3.8b}
\begin{aligned}
&\lim_{t\to\infty} \mathbb{P}\left(\Big\{X^{\rm 2nd}(t+\tau t^{2/3})\geq v_s (t+\tau t^{2/3}) + s t^{1/3}\Big\}\right) \\
&=\mathbb{P}\left( \{{\mathcal H}^-((\rho_--\rho_+)\tau)-{\mathcal H}^+((\rho_+-\rho_-)\tau)\geq 2(\rho_+-\rho_-)s\}\right),
\end{aligned}
\end{equation}
where the processes ${\mathcal H}^+$ and ${\mathcal H}^-$ are \emph{independent}.
\end{theorem}
In \cite{FN24}, the theorem is applied to periodic initial conditions with density $\rho_-$ on $x<0$ and density $\rho_+$ on $x>0$. The assumption $(a),(b)$ is true for product Bernoulli measures with density $\rho_-$ and $\rho_+$, since they only concern the hydrodynamic behavior of the system; also see Remark (3.4) in \cite{FN24}.

For condition (c), we replace with following assumption:
\begin{itemize}
    \item[(c')] There exist limiting processes ${\mathcal H}^-(u)$ and ${\mathcal H}^+(u)$ whose distribution is continuous in $u$, such that
\begin{equation}\label{eq2.3newversion}
\begin{aligned}
&\lim_{t\to\infty} \frac{h^-(t,v_s t+u t^{2/3}-\tfrac{N_r+N_l}{\rho_+-\rho_-})-{\bf h}^-(t,u)-\frac{(2\rho_--1)N_r+(2\rho_+-1)N_l}{\rho_+-\rho_-}}{-t^{1/3}}={\mathcal H}^-(u),\\
&\lim_{t\to\infty} \frac{\tilde h^+(t,v_s t+u t^{2/3}-\tfrac{N_r+N_l}{\rho_+-\rho_-})-{\bf h}^+(t,u)-\frac{(2\rho_--1)N_r+(2\rho_+-1)N_l}{\rho_+-\rho_-}}{-t^{1/3}}={\mathcal H}^+(u),
\end{aligned}
\end{equation}
where ${\bf h}^+(t,u),{\bf h}^-(t,u)$ are defined in the same way as in the original assumption. $N_l,N_r$ are defined in \eqref{eqn:NlNrdef}.
\end{itemize}
Under this assumption, the conclusion about the position of the second class particle is the following.
\begin{equation}
\begin{aligned}
&\lim_{t\to\infty} \mathbb{P}\left(\Big\{X^{\rm 2nd}(t+\tau t^{2/3})\geq v_s (t+\tau t^{2/3}) -\tfrac{N_r+N_l}{\rho_+-\rho_-}+ s t^{1/3}\Big\}\right) \\
&=\mathbb{P}\left( \{{\mathcal H}^-((\rho_--\rho_+)\tau)-{\mathcal H}^+((\rho_+-\rho_-)\tau)\geq 2(\rho_+-\rho_-)s\}\right).
\end{aligned}
\end{equation}
The proof of the new version of the theorem is the same as the proof in \cite{FN24}, therefore, we will omit it here.
The first limit in \eqref{eq2.3newversion} is exactly the same as in $K^l$. The second limit is essentially the same as the proof for $K^r$. The difference is that in equation \eqref{rformula}, the sum is from $0$ to $n-1$, which is ideal since in equation \eqref{eqn:extendto0} we actually need to extend it. Another difference is that now the initial condition on the left of origin is fully stacked, rather than being density $\rho_-$, this is also irrelevant, which can be seen from our proof of Proposition \eqref{prop:mainconvergencetheoremright}. Thus, we derive that ${\mathcal H}^+(\bx),{\mathcal H}^-(\bx)$ are $\Aistat(\tilde{\bx}_+),\Aistat(\tilde{\bx}_-)$, which is the result we want to show.

\section{Appendix}
In this section, we prove the trace norm convergence for the kernel. For the purpose of the convergence of equation \eqref{eqn:117} in the proof of Proposition \eqref{rightpiececonvergence}, we need a scaling on $\bx$ which is slightly more general than the scaling in Proposition \eqref{mainconvergencetheoremleft}. Notice that when $\lambda_1=0,\lambda_2 =1$, it reduces to the scaling in Proposition \eqref{mainconvergencetheoremleft}. When $\lambda_1=1,\lambda_2=0$, it reduces to the scaling in \cite{bfp2010}.
\begin{theorem}\label{thm:convergenceandbound}
    Assume that we scale the variables in the following way. 
    \begin{equation}
    \begin{split}
    &n_- = \rho_-^2\varepsilon^{-3/2}\bt+\rho_-\rho_+\varepsilon^{-1}\lambda_2\bx-\rho_-\varepsilon^{-1}\lambda_1\bx-\varepsilon^{-1/2}\bs, t = \varepsilon^{-3/2}\bt+\varepsilon^{-1}\lambda_2\bx\\
    &a = (1-\rho_--\rho_+)(\varepsilon^{-3/2}\bt+\varepsilon^{-1}\lambda_2\bx)+\varepsilon^{-1}\lambda_1\bx-\frac{N_r+N_l}{\rho_+-\rho_-}.
    \end{split}
    \end{equation}
    If we set $x_1 = \varepsilon^{-3/2}(\rho_+-\rho_-)\bt-\frac{N_r+N_l}{\rho_+-\rho_-}+\varepsilon^{-1/2}\chi_-^{-1/3}(1-\rho_-)\bv, x_2 = a+\varepsilon^{-1/2}\chi_-^{-1/3}(1-\rho_-)\bu$. As $\varepsilon \to 0$,
    \begin{equation}
    \begin{split}
        \varepsilon^{-1/2}\chi_-^{-1/3}(1-\rho_-) S_{-t,-n_-}(x_1,x_2) \rightarrow \mathbf{S}_{-\bt,\bx}(\bv,\bu)\\
        \varepsilon^{-1/2}\chi_-^{-1/3}(1-\rho_-)\overline{S}_{-t,n_-}(x_1,x_2) \rightarrow \overline{\mathbf{S}}_{-\bt,-\bx}(\bv,\bu)
    \end{split}
    \end{equation}
    If $\varepsilon^{1/2}(X^{ll,\varepsilon}_0(\varepsilon^{-1}\by)+\varepsilon^{-1}\by/\rho_--1)\to \mathfrak{h}^-_0(\by),\by>0$ in distribution, where $\mathfrak{h}^-_0(\by) = \mathfrak{h}_0(-\by)$ and $\mathfrak{h}_0(\cdot)$ is a Brownian motion, then
   \begin{equation}
\varepsilon^{-1/2}\chi_-^{-1/3}(1-\rho_-)\overline{S}_{-t,n_-}^{\mathrm{epi}(X^{ll,\varepsilon}_0)}(x_1,x_2) \to \bS_{-\bt,-\bx}^{\mathrm{epi}(- \mathfrak{h}^-_0)}(\bv,\bu).\\
\end{equation}
Moreover, we have the trace norm of $(S_{-t,n_-
})^*\overline{S}_{-t,n_-}^{\mathrm{epi}(X^{ll,\varepsilon}_0)}$ on $L^2(-\infty,a]$ is uniformly bounded in $\varepsilon$.
\end{theorem}
\begin{proof}
We look at the exact contour integral formula for the kernel.
\begin{equation}
\begin{split}
\tfrac{1-\rho_-}{\varepsilon^{1/2}\chi_-^{1/3}}
 S_{-t,-n}(x_1,x_2) &= \tfrac{1-\rho_-}{\varepsilon^{1/2}\chi_-^{1/3}2\pi i}\int_{\Gamma_0}dw_1(\frac{1-\rho_-}{\rho_-})^{n}\frac{(1-\rho_-)^{x_2-x_1}(1-w_1)^{n}}{w_1^{n+1+x_2-x_1}}e^{t(w_1-(1-\rho_-))}\\
&=\tfrac{1-\rho_-}{\varepsilon^{1/2}\chi_-^{1/3}2\pi i}\int_{\Gamma_0}dw_1 e^{\varepsilon^{-3/2}f_1(w_1)+\varepsilon^{-2/2}f_2(w_1)+\varepsilon^{-1/2}f_1(w_1)+f_4(w_1)}
\end{split}
\end{equation}
where 
\begin{equation}
\begin{split}
f_1(w_1)=&\rho_-^2\bt\log(\tfrac{1-w_1}{\rho_-})-(1-\rho_-)^2\bt\log(\tfrac{w_1}{1-\rho_-})+\bt(w_1-(1-\rho_-)),\\
f_2(w_1)=&(\rho_-\rho_+ \lambda_2-\rho_-\lambda_1)\bx\log(\tfrac{1-w_1}{\rho_-})-\big((1-\rho_-)(1-\rho_+)\lambda_2+(1-\rho_-)\lambda_1\big)\bx\log(\tfrac{w_1}{1-\rho_-})\\
&+\lambda_2\bx (w_1-(1-\rho_-)),\\
f_3(w_1)=& -\bs\log(\tfrac{1-w_1}{\rho_-})-(1-\rho_-)\chi_-^{-1/3}(\bu-\bv)\log(\tfrac{w_1}{1-\rho_-})+\bs\log(\tfrac{w_1}{1-\rho_-}),\\
f_4(w_1)=&-\log(w_1).
\end{split}
\end{equation}
For the other part,
\begin{equation}
\begin{split}
\tfrac{1-\rho_-}{\varepsilon^{-1/2}\chi_-^{-1/3}2\pi i}\overline{S}_{-t,n}(x_1,x_2) &= \tfrac{1-\rho_-}{\varepsilon^{1/2}\chi_-^{1/3}2\pi i}\int_{\Gamma_0}dw_2 (\frac{\rho_-}{1-\rho_-})^{n-1}\frac{(1-w_2)^{x_2-x_1+n-1}}{(1-\rho_-)^{x_2-x_1}w_2^n}e^{t(w_2-\rho_-)}\\
&=\tfrac{1-\rho_-}{\varepsilon^{1/2}\chi_-^{1/3}2\pi i}\int_{\Gamma_0}dw_2e^{\varepsilon^{-3/2}g_1(w_2)+\varepsilon^{-2/2}g_2(w_2)+\varepsilon^{-1/2}g_1(w_2)+g_4(w_2)}
\end{split}
\end{equation}
where 
\begin{equation}
\begin{split}
g_1(w_2) &= (1-\rho_-)^2\bt\log(\tfrac{1-w_2}{1-\rho_-})-\rho^2\bt\log(\tfrac{w_2}{\rho_-})+\bt(w_2-\rho_-),\\
g_2(w_2) &= \big((1-\rho_-)(1-\rho_+)\lambda_2+(1-\rho_-)\lambda_1\big)\bx\log(\tfrac{1-w_2}{1-\rho_-})-(\rho_-\rho_+\lambda_2-\rho_-\lambda_1 )\bx\log(\tfrac{w_2}{\rho_-})\\
&+\bx(w_2-\rho_-),\\
g_3(w_2)&=(1-\rho_-)\chi_-^{-1/3}(\bu-\bv)\log(\tfrac{1-w_2}{1-\rho_-})-\bs\log(\tfrac{1-w_2}{1-\rho_-})+\bs \log(\tfrac{w_2}{\rho_-}),\\
g_4(w_2)&=-\log(1-w_2).
\end{split}
\end{equation}
We need to do a standard steep descend analysis. 
\[f_1'(w) = -\frac{t(w-(1-\rho_-))^2}{(1-w)w},\quad g_1'(w) = -\frac{t(w-\rho_-)^2}{(1-w)w}.\]
$f_1(w)$ has a double critical point at $w_1= 1-\rho_-$ and $g_1(w)$ has a double critical point at $w_2= \rho_-$. We make a change of variable $w_1 \to (1-\rho_-)+\varepsilon^{1/2}\chi_-^{1/3}w_1, w_2 \to \rho_-+\varepsilon^{1/2}\chi_-^{1/3}w_2$, the integral becomes:
\begin{equation}
\int_{c_\varepsilon} d\tilde{w_1}e^{F_1(\bt,\bx,\bu,\varepsilon,w_1)},\quad \int_{c_\varepsilon} d\tilde{w_2}e^{F_2(\bt,\bx,\bu,\varepsilon,w_2)}
\end{equation}
where $C_\varepsilon$ is a circle centered at $\varepsilon^{-1/2}$ with radius $\varepsilon^{-1/2}$ and 
 \begin{equation}
 \begin{split}
 &F_1(\bt,\bx,\bu,\varepsilon,w) =(\varepsilon^{-1}\bt+\varepsilon^{-1/2}\lambda_2\bx)\chi_-^{1/3} w\\
 &+(\rho_-^2\varepsilon^{-3/2}\bt+(\rho_-\rho_+ \lambda_2-\rho_-\lambda_1)\varepsilon^{-2/2}\bx+\varepsilon^{-1/2}\bs)\log(1-\frac{1}{\rho_-}\varepsilon^{1/2}\chi_-^{1/3}w)\\
 &-\big[(1-\rho_-)^2\varepsilon^{-3/2}\bt+1+\big((1-\rho_-)(1-\rho_+)\lambda_2+(1-\rho_-)\lambda_1\big)\varepsilon^{-1}\bx\\
 &+\varepsilon^{-1/2}(1-\rho_-)\chi_-^{-1/3}(\bu-\bv)-\varepsilon^{-1/2}\bs\big]\log(1+\tfrac{\varepsilon^{1/2}\chi_-^{1/3}w}{1-\rho_-}),\\
 &F_2(\bt,\bx,\bu,\varepsilon,w) =(\varepsilon^{-1}\bt+\varepsilon^{-1/2}\lambda_2\bx)\chi_-^{1/3} w\\
 &-(\rho_-^2\varepsilon^{-3/2}\bt+(\rho_-\rho_+ \lambda_2-\rho_-\lambda_1) \varepsilon^{-2/2}\bx)\log(1+\varepsilon^{1/2}\chi_-^{1/3}w/\rho_-)\\
 &+\big[ (1-\rho_-)^2\varepsilon^{-3/2}\bt+\big((1-\rho_-)(1-\rho_+)\lambda_2+(1-\rho_-)\lambda_1\big)\bx\varepsilon^{-2/2}\\
 &+\varepsilon^{-1/2}(1-\rho_-)\chi_-^{-1/3}(\bu-\bv)-\varepsilon^{-1/2}\bs-1\big]\log(1-\tfrac{\varepsilon^{1/2}\chi_-^{1/3}w}{1-\rho_-}).
 \end{split}
 \end{equation}
 We deform the path into three piece, make the new contour to be $\tilde{C} = \Gamma_1\cup \Gamma_2 \cup \Gamma_3$, where
 \begin{equation}
  \begin{split}
  \Gamma_1 &= \{re^{i\pi/3}: r\in [0,\varepsilon^{-1/2}]\}\\
  \Gamma_2 &= \{re^{-i\pi/3}: r\in [0,\varepsilon^{-1/2}]\}\\
  \Gamma_3 &= \{\varepsilon^{-1/2}+\varepsilon^{-1/2}e^{i\theta}: \theta \in [-\pi/3,\pi/3]\}
  \end{split}
  \end{equation} 
Taking the expansion of $F_1,F_2$, we have the following. 
\begin{equation}
\begin{split}
F_1(\bt,\bx,\bu,\varepsilon,w) =(-\bu-\chi_-^{-2/3}(1-2\rho_-)\bs) w+\frac{(\rho_--\rho_+)\lambda_2+\lambda_1}{2\chi_-^{1/3}}\bx w^2-\frac{1}{3}\bt w^3+O(\varepsilon^{1/2})\\    
F_2(\bt,\bx,\bu,\varepsilon,w) = (-\bu-\chi_-^{-2/3}(1-2\rho_-)\bs) w-\frac{(\rho_--\rho_+)\lambda_2+\lambda_1}{2\chi_-^{1/3}}\bx w^2-\frac{1}{3}\bt w^3+O(\varepsilon^{1/2})\\
\end{split}
\end{equation}
We show that we can extend the integral on $\Gamma_1$ and $\Gamma_2$ to $\infty$ and the part on $\Gamma_3$ goes to 0. Compute the real part of $F_1$ and $F_2$ on $\Gamma_3$,

\begin{equation}
\begin{split}
&Re(F_1(\bt,\bx,\bu,\varepsilon,\varepsilon^{-1/2}+\varepsilon^{-1/2}e^{i\theta})) =\\ &\frac{\varepsilon^{-3/2}}{2}(-2t-2t\cos(\theta)-( (1-\rho)^2t)\log(\frac{1+\rho^2+2\rho\cos(\theta)}{(1-\rho)^2})\\
&+\rho^2t\log(\frac{2+2\rho+\rho^2+2(1+\rho)\cos(\theta)}{\rho^2})+o(\varepsilon))
\end{split}
\end{equation}
using the fact the $\log(1+x)< x$,
\begin{equation}
\begin{split}
&\frac{\varepsilon
^{-3/2}t}{2}(-2(1-\rho)-2(1-\rho)\cos(\theta)+(\rho^2+O(\varepsilon^{1/2}))(\log(\frac{1+(1-\rho)^2}{\rho^2}+\frac{2(1-\rho)}{\rho}\cos(\theta)))\leq \\
&\frac{\varepsilon
^{-3/2}t}{2}(-2(1-\rho)-2(1-\rho)\cos(\theta)+\rho^2(\frac{2-2\rho}{\rho^2}+\frac{2(1-\rho)}{\rho}\cos(\theta)+O(\varepsilon^{1/2})) = \\
&-\varepsilon^{-3/2}t((1-\rho)^2\cos(\theta)+O(\varepsilon^{1/2}))
\end{split}
\end{equation}
For $-\frac{\pi}{3}<\theta<\frac{\pi}{3}$ and $\varepsilon$ small enough, $-(1-\rho)^2\cos(\theta)+O(\varepsilon^{1/2})< -\kappa$ for some fixed constant $\kappa$, so the integral vanishes in the limit.\\ 
Now, extend the path $\Gamma_1$ and $\Gamma_2$ to infinity, and the extra integral converges to 0 since the real part of the exponent is decreasing along the contour. So, we get the Airy contour. The proof for $F_2$ is the same.
\end{proof}
Up to now, we proved the pointwise convergence of the operator $S_{-t,-n}$ and $\overline{S}_{-t,n}$, and what we want to show is the convergence of the Fredholm determinant. We use the following bounds for the difference of two Fredholm determinants.
\begin{equation}
|\det(I-A)-\det(I-B)|\leq \lVert A-B\rVert_1e^{1+\lVert A\rVert_1+\lVert B\rVert_1}
\end{equation}
where $\lVert \cdot \rVert_1$ is the trace norm. The limiting operator in the trace class is proven in Appendix A in [MQR]. Now we need to get a uniform in the $\varepsilon$ bounds on the trace norm of the discrete kernels.

\begin{proposition}
The trace norm of $(S_{-t,-n})^*\overline{S}^{\mathrm{epi}(X_0)}_{-t,n}$ on $L^2((-\infty,a])$ is uniformly bounded in $\varepsilon$.
\end{proposition}
\begin{proof}
By shifting the height and rescale of the initial condition, we can assume $\bt = 1,\bs =0, \lambda_2=0,\lambda_1=1$. 
We define $S^\varepsilon_{-1,\bx}= S_{-t,-n}$ and $\overline{S}^\varepsilon_{-1,\bx}=\overline{S}_{-t,n}$ to make the variable $\bx$ explicit.
By definition,
\begin{equation}
 \begin{split}
 (S_{-t,-n})^*\overline{S}^{\mathrm{epi}(X_0)}_{-t,n}(z_1,z_2)=\int_{z\in \varepsilon\mathbb{Z},s\in[0,\varepsilon n)}\mathbb{P}_{B_\varepsilon(0)=z}(\tau_\varepsilon \in ds)S_{-1,\bx}(z,z_1)\overline{S}_{-1,-\bx-s}(h_\varepsilon(s),z_2).
 \end{split}
 \end{equation}
 Thus, 
 \begin{equation}
 \begin{split}
 \lVert (S^\varepsilon_{-1,\bx})^*\overline{S}^\varepsilon_{-1,\bx} \rVert_{tr} \leq \int_{z\in \varepsilon\mathbb{Z},s\in[0,\varepsilon n)}\mathbb{P}_{B_\varepsilon(0)=z}(\tau_\varepsilon \in ds)\lVert S^\varepsilon_{-1,\bx}(z,z_1)\overline{S}^\varepsilon_{-1,-\bx-s}(h_\varepsilon(s),z_2) \rVert_{tr}.
 \end{split}
 \end{equation}
 Since $S^\varepsilon_{-1,\bx}(z,z_1), \overline{S}^\varepsilon_{-1,-\bx-s}(h_\varepsilon(s),z_2)$ is a rank one operator, its trace norm is the product of $L^2$ norm,
 \begin{multline}
      \lVert S^\varepsilon_{-1,\bx}(z,z_1)\overline{S}^\varepsilon_{-1,-\bx-s}(h_\varepsilon(s),z_2) \rVert_{tr}\\=(\int_{-\infty}^a|S^\varepsilon_{-1,\bx}(z-z_1)|^2dz_1\int_{\infty}^a |S^\varepsilon_{-1,-\bx-s}(h_\varepsilon(s)-z_2)|^2dz_2)^{1/2}.
 \end{multline}
 We first bound the hitting probability part. From classical theory of Brownian motion, and our assumption on the initial condition that there exists $\alpha,\gamma$, such that
 \begin{equation}
 h^\varepsilon(0,s) \geq -\alpha -\gamma|s|
 \end{equation} uniformly in $\varepsilon$. We have
 \begin{equation}\label{eqn:boundonhittingprob}
 \mathbb{P}(\tau_\varepsilon\leq s) \leq \exp\{-\kappa_2\frac{(z+\alpha+\gamma s)^2}{s}\}.
 \end{equation}
 We now compute the $L^2$ norm of the right-hand side.
 \begin{equation}\label{eqn:166}
 \int_{-\infty}^adz_1 |S^\varepsilon_{-1,\bx}(z-z_1)|^2 =\frac{1}{(2\pi i)^2}\int_{C^0_\varepsilon}dw_1\int_{C^0_\varepsilon}dw_2 \frac{e^{F_1(t,x,z,\varepsilon,w_1)+
F_1(t,x,z,\varepsilon,w_2)}}{w_1+w_2-\varepsilon^{1/2}w_1w_2}
 \end{equation}
 where $C^0_\varepsilon$ is a circle centered at $\varepsilon^{-1/2}$ with radius $\varepsilon^{-1/2}$, with a little perturb at 0 so that 0 is outside the contour. Thus the denominator will not be zero if $\varepsilon$ is small enough, and 
 \begin{multline}\label{eqn:167}
     \int_{-\infty}^adz_2 |\overline{S}^\varepsilon_{-1,\bx}(h_\varepsilon(s)-z_2)| \\= \frac{1}{(2\pi i)^2}\int_{C_\varepsilon^0}dw_1\int_{C_\varepsilon^0}dw_2\frac{e^{F_1(t,x,h_\varepsilon(s),\varepsilon,w_1)+
F_1(t,x,h_\varepsilon(s),\varepsilon,w_2)}(1+\varepsilon^{1/2}w_1)(1+\varepsilon^{1/2}w_2)}{w_1+w_2+\varepsilon^{1/2}w_1w_2}.
 \end{multline}Thus we can write the bound on the trace norm in the following form.
 \begin{equation}\label{eqn:tracenormbound}
\lVert(S^\varepsilon_{-1,\bx})^*\overline{S}^\varepsilon_{-1,\bx}\rVert_1\leq  \int_{z\in \varepsilon\mathbb{Z},s\in[0,\varepsilon n)}\mathbb{P}_{B_\varepsilon(0)=z}(\tau_\varepsilon\in ds)\exp\{F_\varepsilon(\bx,z)+\overline{F}_\varepsilon(-\bx-s,h_\varepsilon(s))\}
\end{equation} 
$e^{F_\varepsilon(\bx,z)} $ and $e^{\overline{F}_\varepsilon(-\bx-s,h_\varepsilon(s))}$ are bounds of equation (\ref{eqn:166}) and equation (\ref{eqn:167}). The general idea is that we want to move the contours in equation (\ref{eqn:166}) and equation (\ref{eqn:167}) to pass through the critical point and deform it to a path such that the real part is decreasing; then we can use the value of integrand at the critical point as our bounds. Of course, it is not necessary for the contours to pass through the critical points since we only want to obtain a bound rather than an asymptotic formula. So, in some cases, we just move the contour to some rather arbitrary places, which make the integral bounded.

Now we analyze the function at the critical point.
The critical points of $F_1$ are:
\begin{equation}
\begin{split}
-\frac{1}{2}(x-\varepsilon^{1/2}u+\varepsilon)\pm \sqrt{(\frac{1}{2}(x-\varepsilon^{1/2}u+\varepsilon))^2+\rho(u-\varepsilon^{1/2})}.
\end{split}
\end{equation}
Thus, we make the change of variable that:
\begin{equation}
x_\varepsilon = \frac{1}{2}(x-\varepsilon^{1/2}u+\varepsilon), \quad u_\varepsilon = \rho(u-\varepsilon^{1/2}).
\end{equation}
Our original function written as:
\begin{equation}\label{formula:Fepsilon}
\begin{split}
F_\varepsilon(t,x_\varepsilon,u_\varepsilon,w) &= (\rho^2\varepsilon^{-3/2}t-2\rho\varepsilon^{-1}x_\varepsilon-\varepsilon^{-1/2}u_\varepsilon)\log(1+\varepsilon^{1/2}w/\rho)\\
&-( (1-\rho)^2\varepsilon^{-3/2}t+2(1-\rho)\varepsilon^{-1}x_\varepsilon-\varepsilon^{-1/2}u_\varepsilon)\log(1-\varepsilon^{1/2}w/(1-\rho))-\varepsilon^{-1}tw
\end{split}
\end{equation}
for which the critical point are:
\begin{equation}
-x_\varepsilon \pm \sqrt{x_\varepsilon^2+u_\varepsilon}.
\end{equation}
Do the similar thing for $F_2$, its critical point are 
\begin{equation}
w_{\pm} = -\frac{1}{2}(x+\varepsilon^{1/2}u+\varepsilon)\pm \sqrt{(\frac{1}{2}(x+\varepsilon^{1/2}u+\varepsilon))^2+\rho(u+\varepsilon^{1/2})}.
\end{equation}
We make the change of variable:
\begin{equation}
\overline{x}_\varepsilon = \frac{1}{2}(x+\varepsilon^{1/2}u+\varepsilon), \quad \overline{u}_\varepsilon = \rho(u+\varepsilon^{1/2}).
\end{equation}
The original function in the new variables are:
\begin{equation}
\begin{split}
\overline{F}_\varepsilon(t,x_\varepsilon,u_\varepsilon,w)&=( (1-\rho)^2\varepsilon^{-3/2}t-2(1-\rho)x\varepsilon^{-1}-\varepsilon^{-1/2}u)\log(1+\varepsilon^{1/2}w/(1-\rho))\\
&-(\rho^2\varepsilon^{-3/2}t+2\rho x\varepsilon^{-1}-\varepsilon^{-1/2}u)\log(1-\varepsilon^{1/2}w/\rho)-\varepsilon^{-1}t w .
\end{split}
\end{equation}
It is not difficult to see $F_\varepsilon$ and $\overline{F}_\varepsilon$ come from the following function:
\begin{equation}
F(t,x,u,w,\rho) = (\rho^2t-2\rho-u)\log(1+w/\rho)-( (1-\rho)^2 t+2 (1-\rho)x-u)\log(1-w/(1-\rho))-t w.
\end{equation}
We look at the value of the integrand at critical point,
\begin{equation}
\begin{split}
F(w_+,x,u) &= \sum_{n=1}^{\infty}\frac{2(-\varepsilon^{1/2}w_+)^{n-1}}{n(n+1)(n+2)}(\frac{1}{(\rho-1)^n}-\frac{1}{\rho^n})(x^3-3xy+2y^{3/2})\\
&+\sum_{n=1}^{\infty}\frac{2(-\varepsilon^{1/2}w_+)^{n-1}(n-1)}{n(n+1)(n+2)}(\frac{1}{(\rho-1)^n}-\frac{1}{\rho^n})(x^2y^{1/2}-2xy+y^{3/2})
\end{split}
\end{equation}
where $y = x^2+u, w_+ =-x+\sqrt{y}$. 
First series is $$\nu_1:=\frac{1}{w^3}(w+(w+\rho-1)^2\log(1-\frac{w}{1-\rho})-(\rho+w)^2\log(1+\frac{w}{\rho})),$$ and the second series is $$\nu_2 := \frac{2}{w^2}(w+(\rho-1)(\rho-1+w)\log(1-\frac{w}{1-\rho})-\rho(\rho+w)\log(1+\frac{w}{\rho})).$$
It can be checked that for all $0<\rho<1$, for $w \in (-\rho,1-\rho)$, $\nu_1-\nu_2<0$.
Also note that 
\begin{equation}
(x^3-3xy+2y^{3/2})-(x^2y^{1/2}-2xy+y^{3/2}) =(\frac{1}{3}y^{3/2}-x^2y^{1/2}+\frac{2}{3}x^{3})+(\frac{2}{3}y^{3/2}-xy+\frac{1}{3}x^3) \geq 0
\end{equation}
Thus the conclusion is that, there exist $C_1>0,C_2>0$, such that 
\begin{equation}\label{criticalpointvalue}
F(w^+,x,u) = -C_1((\frac{1}{3}y^{3/2}-x^2y^{1/2}+\frac{2}{3}x^{3})+(\frac{2}{3}y^{3/2}-xy+\frac{1}{3}x^3))-C_2(x^2y^{1/2}-2xy+y^{3/2})
\end{equation}

We now estimate $F_\varepsilon$ and $\overline{F}_\varepsilon$ over different regions.
\begin{lemma}\label{estimationlemma}
\begin{enumerate}
\item  Suppose $x_\varepsilon^2+u_\varepsilon \geq 0$ and $-x_\varepsilon+\sqrt{x_\varepsilon^2+ u_\varepsilon}\geq (1-\rho)\varepsilon^{-1/2}$, then $F_\varepsilon(t,x,u,\varepsilon,w_1) = -\infty$.
\item Suppose $ u_\varepsilon>0$ and $0 < -x_\varepsilon+\sqrt{x_\varepsilon^2+u_\varepsilon}\leq (1-\rho_-)\varepsilon^{-1/2}\wedge \rho\varepsilon^{-1/2}$, we have $F_\varepsilon(x,u)\leq \hat{F}_\varepsilon(x_\varepsilon,u_\varepsilon)$. Under the same condition, on $\overline{x}_\varepsilon$ and $\overline{u}_\varepsilon$, $\overline{F}_\varepsilon(x,u)\leq \hat{F}_\varepsilon(\overline{x}_\varepsilon,\overline{u}_\varepsilon)$.
\item Suppose $x_\varepsilon >0,u_\varepsilon<0$, then $F_\varepsilon(x_\varepsilon,u_\varepsilon)\leq C\log(2+|u_\varepsilon|+|x_\varepsilon|)$ .
\end{enumerate}
\end{lemma}
\begin{proof}
\begin{enumerate}
\item This is the case that there is no pole in the integrand, thus the integral just vanishes. If 
$-x_\varepsilon+\sqrt{x_\varepsilon^2+ u_\varepsilon})\geq (1-\rho)\varepsilon^{-1/2}$, we have 
\begin{equation}
u_\varepsilon \geq (1-\rho)^2\varepsilon^{-1}+2(1-\rho) \varepsilon^{-1/2}x_\varepsilon
\end{equation}
which indicate the coefficient of the $\log(1-\varepsilon^{1/2}w/(1-\rho))$ term in \ref{formula:Fepsilon} is positive, thus there is no pole in the $e^{F_\varepsilon}$, which means the contour integral is $0$.
\item In this case, we deform the following contour: starting from the critical point $w^\varepsilon_+$, moving along angle $\pi/4$, until it hit the arc of $C_\varepsilon$, then follow the $C_\varepsilon$ to complete the circle. More precisely, $\Gamma = \gamma_1+\gamma_2$, where $\gamma_1 = \{w^\varepsilon_++re^{i\pi/4}: r\in [0,c\varepsilon^{-1/2})\},\gamma_2= \{(1-\rho)\varepsilon^{-1/2}+\varepsilon^{-1/2}e^{i\theta}:\theta\in [0,\pi/4]\}$, constant $c$ is the value such that $|w^\varepsilon_++c\varepsilon^{-1/2}e^{i\pi/4}-(1-\rho)\varepsilon^{-1/2}|= \varepsilon^{-1/2}$.\\
We first check the real part is decreasing along $\gamma_1$. We have
\begin{equation}
\begin{split}
&\partial_r Re[\varepsilon^{-3/2}F(\varepsilon^{1/2}F(\varepsilon^{1/2}(w_+^\varepsilon+re^{i\pi/4}),\varepsilon^{1/2}x_\varepsilon,\varepsilon u_\varepsilon))]\\
&= \frac{-2r^2(\rho(1-\rho)-\varepsilon^{1/2}(1-2\rho)(w+2x)+\varepsilon(2r^2+4r w +3w^2+4(r+w)x))}{((\rho+\varepsilon^{1/2}(r+w))^2+\varepsilon r^2)((1-\rho+\varepsilon^{1/2}(r+w))^2+\varepsilon r^2))}
\end{split}
\end{equation}
It is easy to see that $w_+^\varepsilon + 2 x_\varepsilon > 0$ since $u>0$. The order $\varepsilon$ term is always positive. For order $\varepsilon^{1/2}$ terms, if $1-2\rho<0$, then the whole numerator is easily seen to be negative; if $1-2\rho >0$, if we have $w+2x<\rho\varepsilon^{-1/2}$, then we have 
\begin{equation}
\rho(1-\rho) - \varepsilon^{1/2}(1-2\rho)(w_\varepsilon+2x_\varepsilon)>\rho^2
\end{equation}
so the numerator is still positive.\\
The reason for $w_+^\varepsilon+2x_\varepsilon<\varepsilon^{-1/2}\rho$ is that it is equivalent to 
\begin{equation}
x_\varepsilon^2+u_\varepsilon < (\rho\varepsilon^{-1/2}-x_\varepsilon)^2
\end{equation}
which is equivalent to 
\begin{equation}
\rho (u -\varepsilon^{1/2}) <\rho^2\varepsilon^{-1}-(x-\varepsilon^{1/2}u+\varepsilon)\rho\varepsilon^{-1/2}
\end{equation}
Since $x$ is of constant order, thus the equality is always true if we make $\varepsilon$ small enough.\\
Then we check the real part is decreasing along $\gamma_2$,
\begin{equation}
\begin{split}
&\partial_\theta(F( (1-\rho+e^{i\theta},\varepsilon^{1/2}x_\varepsilon,\varepsilon u_\varepsilon))))\\
&=2\varepsilon^{-3/2}(2-\rho^2+2\cos(\theta)+\varepsilon u+2\rho\varepsilon^{1/2}x)\tan(\theta/2)
\end{split}
\end{equation}
which is clearly positive.\\
\item This is the part where the estimate is the same as \cite{MQR2021}. We move the contour to $q := (1+|u_\varepsilon+|x_\varepsilon|)^{-1}\in (0,\varepsilon^{-1/2})$, then move in the vertical direction until it hit $\Gamma$ contour in previous case, then follow the $\Gamma$ contour. Along the vertical line, 
\begin{equation}
\begin{split}
&\partial_r Re[\varepsilon^{-3/2}F(\varepsilon^{1/2}(q+ir),\varepsilon^{1/2}x_\varepsilon,\varepsilon u_\varepsilon))]\\
&=\frac{-r(2\rho(1-\rho)(q+x)+\varepsilon^{1/2}(q^2+r^2+u)(1-2\rho)+2\varepsilon(q^2x+r^2x-2 q u)}{((\rho+\varepsilon^{1/2}q)^2+\varepsilon r^2)((1-\rho+\varepsilon^{1/2}q)^2+\varepsilon r^2))}
\end{split}
\end{equation}
We focus on the case that $u$ is very negative. If $(1-2\rho)<0$, the numerator is always negative. If $1-2\varphi>0$, we compare it with the first term. Recall, $x_\varepsilon = \frac{1}{2}(x-\varepsilon^{1/2}u+\varepsilon), u_\varepsilon = \rho(u-\varepsilon^{1/2})$, so when $u \asymp -\varepsilon^{-1/2}$,
\begin{equation}
2\rho(1-\rho)(q+x)+\varepsilon^{1/2}(q^2+r^2+u) = -\rho(1-\rho)u+\rho(1-2\rho)u+o(1)>0
\end{equation}
\end{enumerate}

\end{proof}

Let $\overline{z} = \bx(1-\rho)\varepsilon^{-1/2}+(1-\rho)^2\varepsilon^{-1}+\varepsilon^{1/2}$.
When $z \geq \overline{z}$, we use Lemma \ref{estimationlemma}(i) so that the whole integral is 0. When $-1\leq z<\overline{z}$, we use \ref{estimationlemma}(ii),
\begin{equation}\label{eqn:188}
    F_\varepsilon(\bx,z)< F(w_\varepsilon^+,x_\varepsilon,u_\varepsilon).
\end{equation}
Choosing the most significant terms from \ref{criticalpointvalue}, we have
\begin{equation}\label{eqn:189}
    F_\varepsilon(x_1,z)\leq C_2-\frac{C_1}{3}|z|^{3/2}.
\end{equation}
And when $z < -1$, we do not need a good bound, since we have the estimate from $\mathbb{P}_z(\tau_\varepsilon<s)$, thus we use the lemma \ref{estimationlemma}(iii), to find a $C_3<\infty$ such that $F_\varepsilon(\bx,z)\leq C_3(1+\log(|z|))$. Also note that the value $-1$ is arbitrary, any negative constant value can work.

Now we estimate $\overline{F}_\varepsilon(-\bx-s,h_\varepsilon(s))$. If $z > g(0)$, then $\tau = 0$ and $h_\varepsilon(0)=z$. Thus, the integral reduces to
\[\int_{z\in \varepsilon \mathbb{Z}}\exp{\{F_{\varepsilon}(\bx,z)+\overline{F}_\varepsilon(-\bx,z)}\}\]
Since $F_\varepsilon$ and $\overline{F}_\varepsilon$ are basically the same function and $\bx$ is some constant, we have the same bound.

For the case the $z \leq g(0)$, first using 
the linear lower bound for the initial configuration and upper bound for the random walk, we have:
\begin{equation}
-\alpha-\gamma s\leq h_\varepsilon(s)\leq z+\varepsilon^{-1/2}s
\end{equation}
 it is easy to see that $\overline{x}_\varepsilon\leq 0$ for $\varepsilon$ small enough. Notice that $\overline{x}_\varepsilon^2+h_\varepsilon(s)$ will only be less $0$ when $h_\varepsilon(s)<0$ and $s<s_0$ for some constant $s_0$ which does not depend on $\varepsilon$, this is because $\overline{x}_\varepsilon^2 = s^2 +O(1)$ and $h_\varepsilon(s) \geq -\alpha-\gamma s$. In the region that $0\leq s \leq s_0$, $\overline{F}(-\bx-s,h_\varepsilon(s))<C$ for $\varepsilon$ small enough. Thus, we can assume that $\overline{x}_\varepsilon^2+h_\varepsilon(s)\geq 0$, together with the fact that $s \leq \varepsilon n$, the condition of Lemma (\ref{estimationlemma})(ii) is satisfied. Using equation(\ref{criticalpointvalue}) We have 
 \[\overline{F}_\varepsilon(-\bx-s,b)\leq \hat{F}_\varepsilon(\overline{x}_\varepsilon,\overline{u}_\varepsilon)\leq C(\overline{x}_\varepsilon h_\varepsilon(s)+\frac{2}{3}\overline{x}_\varepsilon^3-\frac{2}{3}(\overline{x}_\varepsilon^2+h_\varepsilon(s))^{3/2}.\]
 Using the bounds on $h_\varepsilon(s)$ we can see that
\begin{equation}\label{eqn:191}
    \overline{F}_\varepsilon(-\bx-s,b)\leq -C_4 s^3+C_5
\end{equation}
for some $C_4,C_5 >0$.

Using equation (\ref{eqn:188}), (\ref{eqn:189}), (\ref{eqn:191}) and (\ref{eqn:boundonhittingprob}),
the trace norm (\ref{eqn:tracenormbound}) can be bounded by
\begin{multline*}
    \int_{-\infty}^{-\overline{\alpha}}dz\int_0^\infty ds e^{-\kappa_2\frac{(z+\alpha+\gamma s)^2}{s}+C_3\log|z|-C_4s^3+C_5}\\
    +\int_{-\overline{\alpha}}^{g(0)}dz\int_0^\infty ds e^{-\frac{C_1}{3}|z|^{3/2}-C_4s^3+C_5} +
\int_{g(0)}^{\infty}dz e^{-\frac{2C_1}{3}|z|^{3/2}}
\end{multline*}
which is convergent.
\end{proof}

\bibliography{shocks}

\bibliographystyle{alpha}

\end{document}